\documentclass[12pt]{amsart}
\usepackage{enumerate}
\usepackage{hyperref}
\usepackage{graphicx,color}
\usepackage{cite}
\usepackage{amsthm, amsmath, amssymb, mathtools}
\usepackage[top=45truemm, bottom=45truemm, left=30truemm, right=30truemm]{geometry}

\renewcommand{\epsilon}{\varepsilon}
\renewcommand{\limsup}{\varlimsup}
\renewcommand{\liminf}{\varliminf}
\renewcommand{\Re}{\operatorname{Re}}
\newcommand{\abs}[1]{\left|#1\right|}

\newcounter{count}
\newcommand{\num}{\stepcounter{count}\the\value{count}}

\numberwithin{equation}{section}

\newtheorem{theorem}{Theorem}[section]
\newtheorem{lemma}[theorem]{Lemma}
\newtheorem{corollary}[theorem]{Corollary}
\newtheorem{definition}[theorem]{Definition}
\newtheorem{proposition}[theorem]{Proposition}

\newtheorem{question}[theorem]{Question}

\theoremstyle{definition}

\newtheorem{remark}[theorem]{Remark}

\newcommand{\cB}{\mathcal{B}}
\newcommand{\cC}{\mathcal{C}}

\newcommand{\cF}{\mathcal{F}}

\newcommand{\cH}{\mathcal{H}}
\newcommand{\cI}{\mathcal{I}}

\newcommand{\cP}{\mathcal{P}}
\newcommand{\cQ}{\mathcal{Q}}

\newcommand{\cS}{\mathcal{S}}

\newcommand{\cV}{\mathcal{V}}
\newcommand{\cW}{\mathcal{W}}

\newcommand{\cLE}{\mathcal{LE}}

\title[Distributions of Finite Sequences Represented by Polynomials]{Distributions of Finite Sequences Represented by Polynomials in Piatetski-Shapiro Sequences}

\author[K.\ Saito]{Kota Saito}
\address{Kota Saito\\
Graduate School of Mathematics\\ Nagoya University\\ Furo-cho\\ Chikusa-ku\\ Nagoya\\ 464-8602\\ Japan}
\curraddr{}
\email{m17013b@math.nagoya-u.ac.jp}

\author[Y.\ Yoshida]{Yuuya Yoshida}
\address{Yuuya Yoshida\\
Graduate School of Mathematics\\ Nagoya University\\ Furo-cho\\ Chikusa-ku\\ Nagoya\\ 464-8602\\ Japan}
\curraddr{}
\email{m17043e@math.nagoya-u.ac.jp}



\begin{document}
\maketitle

\begin{abstract}
By using the work of Frantzikinakis and Wierdl, we can see that 
for all $d\in\mathbb{N}$, $\alpha\in(d,d+1)$, and integers $k\ge d+2$ and $r\ge1$, 
there exist infinitely many $n\in\mathbb{N}$ 
such that the sequence $(\lfloor{(n+rj)^\alpha}\rfloor)_{j=0}^{k-1}$ is represented as 
$\lfloor{(n+rj)^\alpha}\rfloor=p(j)$, $j=0,1,\ldots,k-1$, 
by using some polynomial $p(x)\in\mathbb{Q}[x]$ of degree at most $d$.
In particular, the above sequence is an arithmetic progression when $d=1$.
In this paper, we show the asymptotic density of such numbers $n$ as above.
When $d=1$, the asymptotic density is equal to $1/(k-1)$.
Although the common difference $r$ is arbitrarily fixed in the above result, 
we also examine the case when $r$ is not fixed.
Most results in this paper are generalized by using functions belonging to Hardy fields.
\end{abstract}

\section{Introduction}\label{intro}

\textit{Notations.}---%
A subset $A$ of $\mathbb{N} \coloneqq \{1,2,\ldots\}$ is naturally identified with 
a strictly increasing sequence of $\mathbb{N}$, and \textit{vice versa}.
For an interval $\cI$ of $\mathbb{R}$, the notation $\cI_{\mathbb{Z}}$ denotes the set $\cI\cap\mathbb{Z}$; 
for $x\in\mathbb{R}$, the notation $\lfloor{x}\rfloor$ (resp.\ $\lceil{x}\rceil$) denotes 
the greatest (resp.\ least) integer $\le x$ (resp.\ $\ge x$).

\textit{Motivation and introduction.}---%
For non-integral $\alpha>1$, 
the sequence $\mathrm{PS}(\alpha) \coloneqq (\lfloor{n^\alpha}\rfloor)_{n=1}^\infty$ 
is called a \textit{Piatetski-Shapiro sequence}, 
which has rich mathematical structures, 
e.g., prime numbers \cite{Piatetski-Shapiro,Rivat-Wu,Rivat-Sargos}, square-free numbers \cite{Baker,Cao-Zhai1,Cao-Zhai2}, 
cube-free numbers \cite{Deshouillers2}, and numbers congruent to $c$ modulo $m$ \cite{Deshouillers1,Morgenbesser} 
(for details, see the end of this section).
As another mathematical structure, 
we consider the sets 
\[
\cP_{k,d} \coloneqq \biggl\{
\begin{array}{l}
	(a(n))_{n=0}^{k-1}\subset\mathbb{N}\\
	\text{ strictly increasing}
\end{array}
: (\Delta_1^d a(n))_{n=0}^{k-d-1}\text{ is a constant sequence} \biggr\}
\]
with integers $d\ge1$ and $k\ge d+2$, 
where $\Delta_r^d$ is the $d$-th order difference operator, i.e., 
\[
\Delta_r a(n) \coloneqq a(n+r) - a(n),\quad
\Delta_r^m \coloneqq \Delta_r\circ\Delta_r^{m-1}\quad(m=2,3,\ldots).
\]
A sequence $(a(n))_{n=0}^{k-1}$ of $\mathbb{N}$ belongs to $\cP_{k,d}$ 
if and only if $(a(n))_{n=0}^{k-1}$ is represented as 
$a(n)=p(n)$, $n\in[0,k)_{\mathbb{Z}}$, 
by using some polynomial $p(x)\in\mathbb{Q}[x]$ of degree at most $d$.
When $d=1$, 
a sequence belonging to $\cP_{k,1}$ is called an \textit{arithmetic progression of length $k$ (for short, $k$-AP)}, 
which is a simple additive structure of $\mathbb{N}$.

We find that every $\mathrm{PS}(\alpha)$ with $\alpha\in(d,d+1)$ and $d\in\mathbb{N}$ 
contains infinitely many sequences belonging to $\cP_{k,d}$ 
by using the work of Frantzikinakis and Wierdl \cite{FW}.
Precisely speaking, 
for all $d\in\mathbb{N}$, $\alpha\in(d,d+1)$, and integers $k\ge d+2$ and $r\ge1$, 
there exist infinitely many $n\in\mathbb{N}$ 
such that $(\lfloor{(n+rj)^\alpha}\rfloor)_{j=0}^{k-1}$ belongs to $\cP_{k,d}$.
However, the asymptotic density of such numbers $n$ was not known.
In this paper, we show the asymptotic density, 
which can be expressed as the volume of a convex set of $\mathbb{R}^{d+1}$.

\begin{theorem}\label{main00}
	Let $d\in\mathbb{N}$.
	For all $\alpha\in(d,d+1)$ and all integers $k\ge d+2$ and $r\ge1$, 
	\begin{equation*}
		\lim_{N\to\infty} \frac{1}{N} \#\{ n\in[1,N]_{\mathbb{Z}} : (\lfloor{(n+rj)^\alpha}\rfloor)_{j=0}^{k-1}\in\cP_{k,d} \}
		= \mu(\cC_{k,d+1}), 
	\end{equation*}
	where $\mu$ denotes the Lebesgue measure on $\mathbb{R}^{d+1}$ and 
	the convex set $\cC_{k,d+1}$ of $\mathbb{R}^{d+1}$ is defined as 
	\begin{equation}
		\cC_{k,d+1} = \biggl\{ (y_i)_{i=0}^d\in\mathbb{R}^{d+1} : 0\le y_0<1,\ 0\le\sum_{i=0}^d \binom{j}{i}y_i<1\ (\forall j\in[1,k)_{\mathbb{Z}}) \biggr\}.
		\label{Ckd}
	\end{equation}
	Also, $\mu(\cC_{k,d+1})$ is bounded below by $1/\prod_{i=1}^d \binom{k-1}{i}$.
\end{theorem}

Note that for integers $n,l\ge0$ the binomial coefficient $\binom{n}{l}$ is defined as 
\[
\binom{n}{l}=\frac{(n)_l}{l!},
\]
where $(x)_l$ denotes the falling factorial: 
$(x)_l=x(x-1)\cdots(x-l+1)$ if $l\in\mathbb{N}$, 
and $(x)_l=1$ if $l=0$.
Hence, $\binom{n}{l}=0$ if $0\le n<l$.
From the last sentence in Theorem~\ref{main00}, 
it follows that $\mu(\cC_{k,d+1})$ is positive.
When $d=1$, we can rewrite Theorem~\ref{main00} more simply.

\begin{corollary}\label{main00'}
	For all $\alpha\in(1,2)$ and all integers $k\ge3$ and $r\ge1$, 
	\begin{equation}
		\lim_{N\to\infty} \frac{1}{N} \#\{ n\in[1,N]_{\mathbb{Z}} : (\lfloor{(n+rj)^\alpha}\rfloor)_{j=0}^{k-1}\in\cP_{k,1} \}
		= \frac{1}{k-1}. \label{eq01'}
	\end{equation}
\end{corollary}
\begin{proof}[Proof of Corollary~$\ref{main00'}$ assuming Theorem~$\ref{main00}$]
	Since the convex set $\cC_{k,2}$ is equal to 
	\[
	\{ (y_0,y_1)\in\mathbb{R}^2 : 0\le y_0<1,\ 0\le y_0+(k-1)y_1<1 \},
	\]
	Theorem~\ref{main00} implies this corollary.
\end{proof}

The lower bound $1/\prod_{i=1}^d \binom{k-1}{i}$ of $\mu(\cC_{k,d+1})$ 
is not equal to $\mu(\cC_{k,d+1})$ in general, 
although the two values are equal to each other when $d=1$.
Also, the volume $\mu(\cC_{k,d+1})$ can be computed by using a convex hull algorithm if necessarily.
The definition of Piatetski-Shapiro sequences uses the function $x^\alpha$, 
which is generalized to a function $f$ with certain properties (Theorems~\ref{main0} and \ref{main1}).
Moreover, we estimate the convergence speed of \eqref{eq01'} (Theorem~\ref{main2}).

Theorem~\ref{main00} and Corollary~\ref{main00'} can be regarded as 
the case when the common difference $r$ is fixed.
We next consider the case when the common difference $r$ is not fixed.

\begin{theorem}\label{main01}
	Let $d\in\mathbb{N}$.
	For all $\alpha\in(d,d+1)$ and all integers $k\ge d+2$, 
	there exist $A_{\alpha,k}, B_{\alpha,k}>0$ and $N_{\alpha,k}\in\mathbb{N}$ such that 
	for all integers $N\geq N_{\alpha,k}$, 
	\begin{equation}
		A_{\alpha,k}N^{2-\alpha/(d+1)}
		\leq \#\{ P\subset[1,N]_{\mathbb{Z}} : P\in\cP_{1,d},\ (\lfloor{n^\alpha}\rfloor)_{n\in P}\in\cP_{k,d} \}
		\leq B_{\alpha,k}N^{2-\alpha/(d+1)}. \label{eq02}
	\end{equation}
\end{theorem}

Since the number of $k$-APs contained in $[1,N]_{\mathbb{Z}}$ is about $N^2/2(k-1)$, 
the asymptotic density of the set in \eqref{eq02} is zero.
We give explicit values of $A_{\alpha,k}$ and $B_{\alpha,k}$ in Section~\ref{proofs}.

\textit{Related work on Piatetski-Shapiro sequences.}---%
Piatetski-Shapiro \cite{Piatetski-Shapiro} proved that 
for every $\alpha\in(1,12/11)$, $\mathrm{PS}(\alpha)$ contains infinitely many prime numbers.
For non-integral $\alpha>1$, prime numbers of the form $\lfloor{n^\alpha}\rfloor$ are called \textit{Piatetski-Shapiro primes}.
It is known that the range $(1,12/11)$ can be improved to $(1,243/205)$ \cite{Rivat-Wu}. 
Actually, Piatetski-Shapiro \cite{Piatetski-Shapiro} proved a stronger statement, namely, 
the prime number theorem on Piatetski-Shapiro sequences: 
for all $\alpha\in(1,12/11)$, 
\[
\#\{ n\in[1,x]_{\mathbb{Z}} : \lfloor{n^\alpha}\rfloor\text{ is prime} \}
\sim \frac{x}{\alpha\log x} \quad (x\to\infty).
\]
Rivat and Sargos \cite{Rivat-Sargos} improved the range $(1,12/11)$ to $(1,2817/2426)$.
Similar formulas are also known for square-free $\lfloor{n^\alpha}\rfloor$ \cite{Baker,Cao-Zhai1,Cao-Zhai2}, 
cube-free $\lfloor{n^\alpha}\rfloor$ \cite{Deshouillers2}, 
and $\lfloor{n^\alpha}\rfloor \equiv c\bmod{m}$ \cite{Deshouillers1,Morgenbesser} 
when $\alpha$ lies in certain ranges.
Also, Frantzikinakis and Wierdl \cite{FW} showed the following statement: 
Let $d\in\mathbb{N}$ and $\alpha\in(d,d+1)$, and let $k\ge d+1$ and $r\ge1$ be integers; 
Then, for every sufficiently large $m\in\mathbb{N}$, 
there exists a polynomial $p(x)\in\mathbb{Z}[x]$ of degree at most $d-1$ such that 
$\{r(mn^d+p(n)) : n\in[0,k)_{\mathbb{Z}}\} \subset \mathrm{PS}(\alpha)$.

In general, if a subset $A$ of $\mathbb{N}$ has \textit{positive upper density}, i.e., 
\[
\limsup_{N\to\infty} \frac{\#(A\cap[1,N])}{N} > 0,
\]
then the famous theorem by Szemer\'edi \cite{szemeredi} guarantees that 
the set $A$ contains arbitrarily long APs.
However, the asymptotic density of $\mathrm{PS}(\alpha)$ is zero for every $\alpha>1$, 
and thus one cannot directly apply Szemer\'edi's theorem to any Piatetski-Shapiro sequences.
Nevertheless, the set $\{\lfloor{n^\alpha}\rfloor : n\in A\}$ contains arbitrarily long APs 
for every $\alpha\in(1,2)$ and every $A\subset\mathbb{N}$ with positive upper density \cite{SY}.
In this way, it has been studied whether a subset of $\mathbb{N}$ with asymptotic density zero contains long APs or not.
The set of all prime numbers is a famous one that contains arbitrarily long APs and has asymptotic density zero \cite{GreenTao1}.
Recently, some researchers studied APs of Piatetski-Shapiro primes with fixed exponent.
Mirek \cite{Mirek} proved that for every $\alpha\in(1,72/71)$, 
the set of all Piatetski-Shapiro primes with exponent $\alpha$ contains infinitely many $3$-APs. 
Li and Pan \cite{Li-Pan} claimed that for every integer $k\ge3$, 
there exists $\alpha_k>1$ such that for every $\alpha\in(1,\alpha_k)$ 
the set of all Piatetski-Shapiro primes with exponent $\alpha$ contains infinitely many $k$-APs.

Looking at $3$-APs in the viewpoint of Diophantine equations, 
one can characterize a $3$-AP $\{x<y<z\}$ by the equation $x+z=2y$.
Hence, for every $\alpha\in(1,2)$, the equation  $x+z=2y$ has infinitely many solutions 
even if $x<y<z$ are elements of $\mathrm{PS}(\alpha)$.
Thus, it is natural to consider another Diophantine equation.
Glasscock \cite{Glasscock} mentioned that for every $\alpha\in(1,2)$, 
the equation $x+y=z$ has infinitely many solutions in $\mathrm{PS}(\alpha)$.
Also, he showed that 
if the equation $y=ax+b$ with real $a>b\ge0$ has infinitely many solutions in $\mathbb{N}$, 
then, for Lebesgue-a.e.\ $\alpha\in(1,2)$, 
the equation $y=ax+b$ has infinitely many solutions in $\mathrm{PS}(\alpha)$ \cite{Glasscock}.

\section{Main results and Hardy fields}\label{results}

First, let us define asymptotic notations.
Suppose that for some $x_0>0$, 
complex-valued functions $f$, $f_1$ and $f_2$ and positive-valued functions $g$, $g_1$ and $g_2$ 
are defined on the interval $[x_0,\infty)$.
We write 
\begin{itemize}
	\item
	$f_1(x)=f_2(x)+O(g(x))$ as $x\to\infty$ 
	if there exists $C>0$ such that $|f_1(x)-f_2(x)|\le Cg(x)$ for every sufficiently large $x>0$;
	\item
	$g_1(x) \asymp g_2(x)$ as $x\to\infty$ 
	if $g_1(x)\ll g_2(x)$ and $g_2(x)\ll g_1(x)$ as $x\to\infty$;
	\item
	$g_1(x)\prec g_2(x)$ as $x\to\infty$ 
	if $g_2(x)/g_1(x)$ diverges to positive infinity as $x\to\infty$.
\end{itemize}
The symbol ``$x\to\infty$'' is often omitted.
Also, we often use the following non-asymptotic notation: 
\begin{itemize}
	\item
	$f(x)\ll g(x)$ for all $x\ge x_0$ if there exists $C>0$ such that $|f(x)|\le Cg(x)$ for all $x\ge x_0$.
\end{itemize}
The above constants $C$ are called implicit constants.
When implicit constants depend on parameters $a_1,\ldots,a_n$, 
we often write $f_1(x)=f_2(x)+O_{a_1,\ldots,a_n}(g(x))$, 
$f(x) \ll_{a_1,\ldots,a_n} g(x)$, and $g_1(x) \asymp_{a_1,\ldots,a_n} g_2(x)$ 
to emphasize the dependence.

Next, we describe \textit{Hardy fields} 
which are convenient to extend Piatetski-Shapiro sequences to more general ones.
Let $\cB$ be the set of all real-valued functions on intervals $[x_0,\infty)$, 
where the real numbers $x_0$ depend on the functions.
The set $\cB$ forms a ring under the induced addition and multiplication 
by the following equivalence relation: 
two functions $f_1,f_2\in\cB$ are equivalent to each other 
if and only if there exists $x'_0\in\mathbb{R}$ such that $f_1(x)=f_2(x)$ for all $x\ge x'_0$.
Using this equivalence relation, 
we define Hardy fields as follows.

\begin{definition}
	A subfield of the ring $\cB$ closed under differentiation is called a Hardy field.
	We denote by $\cH$ the union of all Hardy fields.
\end{definition}

The notion of Hardy fields was first introduced by Bourbaki \cite{Bourbaki}, 
and has been used in analysis, e.g., 
differential equations \cite{Bo2,Bo3,Bo4,Ro1,Ro2}, difference and functional equations \cite{Bo5,Bo6}, 
and uniform distribution modulo $1$ \cite{Bo1,BKS,Frantzikinakis}.
The set $\cH$ is so rich that $\cH$ contains the set $\cLE$ of all \textit{logarithmico-exponential functions}.
A logarithmico-exponential function, which was introduced by Hardy \cite{Hardy1,Hardy2}, 
is defined by a finite combination of 
the ordinary algebraic symbols (viz.\ $+,-,\times,\div$) and the functional symbols $\log(\cdot)$ and $\exp(\cdot)$ 
operating on a real variable $x$ and on real constants.
For instance, the function $x^\alpha=e^{\alpha\log x}$ belongs to $\cLE$ for all $\alpha\in\mathbb{R}$.

The function $x^\alpha$ used in Theorems~\ref{main00} and \ref{main01} 
is generalized to a function $f\in\cH$ with $x^d\log x \prec f(x) \prec x^{d+1}$.
Such a function $f$ satisfies that $f'(x)\ge1$ for every sufficiently large $x>0$, 
since the relation $f(x)\succ x^d\log x$ implies $f'(x)\succ x^{d-1}\log x$ (see Section~\ref{uni-dist}).
From now on, we assume that a differentiable function $f\colon [n_0,\infty)\to\mathbb{R}$ satisfies $\inf_{x\ge n_0} f'(x)\ge1$ 
in order to make the sequence $(\lfloor{f(n)}\rfloor)_{n=n_0}^\infty$ an increasing sequence.
However, this assumption is not essential in any proofs of theorems.

\begin{theorem}\label{main0}
	Let $n_0,d\in\mathbb{N}$, 
	and let $f\colon [n_0,\infty)\to\mathbb{R}$ be a differentiable function in $\cH$ 
	satisfying that 
	\begin{itemize}
		\item[(a1)]
		$x^d\log x \prec f(x) \prec x^{d+1}$;
		\item[(a2)]
		$\inf_{x\ge n_0} f'(x)\ge1$.
	\end{itemize}
	Then, for all integers $k\ge d+2$ and $r\ge1$, 
	\begin{equation}
		\lim_{N\to\infty} \frac{1}{N}
		\#\{ n\in[n_0,N]_{\mathbb{Z}} : (\lfloor{f(n+rj)}\rfloor)_{j=0}^{k-1}\in\cP_{k,d} \}
		= \mu(\cC_{k,d+1}), \label{eq:main0}
	\end{equation}
	where $\mu$ denotes the Lebesgue measure on $\mathbb{R}^{d+1}$ and 
	the convex set $\cC_{k,d+1}$ of $\mathbb{R}^{d+1}$ is defined as \eqref{Ckd}.
	Also, $\mu(\cC_{k,d+1})$ is bounded below by $1/\prod_{i=1}^d \binom{k-1}{i}$.
\end{theorem}

\begin{theorem}\label{main1}
	Let $n_0,d\in\mathbb{N}$, 
	and let $f\colon [n_0,\infty)\to\mathbb{R}$ be the same as Theorem~$\ref{main0}$.
	Then, for every integer $k\ge d+2$, 
	\begin{equation}
	\begin{split}
		&\quad \#\{ P\subset[n_0,N]_{\mathbb{Z}} : P\in\cP_{k,1},\ (\lfloor{f(n)}\rfloor)_{n\in P}\in\cP_{k,d} \}\\
		&\asymp_{c(\cdot),k,d} Nf^{(d+1)}(N)^{-1/(d+1)} \quad (N\to\infty).
	\end{split}\label{eq:main1}
	\end{equation}
\end{theorem}

When $d=1$, 
one can apply Theorems~\ref{main0} and \ref{main1} to the following functions: 
\begin{equation}\label{equation=examples}
	x^\alpha,\quad
	x(\log x)^\beta,\quad
	\frac{x^2}{(\log x)^\gamma},\quad
	\frac{x^2}{(\log\log x)^\gamma},
\end{equation}
where $\alpha\in(1,2)$, $\beta>1$ and $\gamma>0$.
Note that all the above functions belong to $\cLE$ and \textit{a fortiori} $\cH$.
Hence, Theorems~\ref{main00} and \ref{main01} are special cases of Theorems~\ref{main0} and \ref{main1}, respectively.
Also, the implicit constants of \eqref{eq:main1} only depend on $c(\cdot)$, $k$ and $d$.
This fact is seen in Section~\ref{proofs} by giving explicit values of the implicit constants.
For special $c(\cdot)$, the explicit values can be simplified, e.g., 
the case when $f(x)=x^\alpha$ with $\alpha\in(d,d+1)$.
For details, see Remarks~\ref{rem:lb} and \ref{rem:ub}.

Finally, let us focus on $\mathrm{PS}(\alpha)$ with $\alpha\in(1,2)$.
Recall that the asymptotic density \eqref{eq01'} is equal to $1/(k-1)$.
However, Corollary~\ref{main00'} does not give us any information about convergence speed.
The convergence speed of \eqref{eq01'} is estimated as follows.

\begin{theorem}\label{main2}
    For all $\alpha\in(1,2)$ and all integers $k\ge3$ and $r\ge1$, 
	\begin{align*}
		&\quad \frac{1}{N}\#\{ n\in[1,N]_{\mathbb{Z}} : (\lfloor{(n+rj)^\alpha}\rfloor)_{j=0}^{k-1}\in\cP_{k,1} \}\\
		&= \frac{1}{k-1} + O_{\alpha,k,r}(F(N)) \quad (N\to\infty),
	\end{align*}
	where 
	\[
	F(x) \coloneqq
	\begin{cases}
		x^{(1-\alpha)/2} & \alpha\in(1,5/4),\\
		x^{(\alpha-3)/14}(\log x)^{1/2} & \alpha\in[5/4,11/6),\\
		x^{(\alpha-2)/6}(\log x)^{1/2} & \alpha\in[11/6,2).
	\end{cases}
	\]
\end{theorem}

Theorem~\ref{main2} gives an upper bound for the convergence speed of \eqref{eq01'}.
We show an extended statement (Proposition~\ref{mainprop}) in Section~\ref{discrepancy}, 
which can be applied to a short interval $[N,N+L)$.
Theorem~\ref{main2} is derived from the extended statement.

So far, we have stated only asymptotic results.
In general, an asymptotic result does not give the information 
how long an interval containing no numbers $n$ in the set in \eqref{eq01'} is.
Hence, we need a non-asymptotic result in order to know such information.
To state a non-asymptotic result, 
let us define the minimum length $L_{\alpha,k,r}(x)$ as 
\[
L_{\alpha,k,r}(x)
= \min\{ y\ge0 : \exists n\in[x,x+y]_{\mathbb{Z}},\ (\lfloor{(n+rj)^\alpha}\rfloor)_{j=0}^{k-1}\in\cP_{k,1} \}
\]
for $\alpha\in(1,2)$, $x\ge1$, and integers $k\ge3$ and $r\ge1$.
The following theorem gives an upper bound for $L_{\alpha,k,r}(x)$.

\begin{theorem}\label{main3}
	For all $\alpha\in(1,2)$ and all integers $k\ge3$ and $r\ge1$, 
	we have $L_{\alpha,k,r}(x)=O_{\alpha,k,r}(x^{2-\alpha})$ as $x\to\infty$.
\end{theorem}

At glance, the growth rate $O_{\alpha,k,r}(x^{2-\alpha})$ is strange 
because it becomes smaller when $\alpha$ increases.
However, for all $\alpha\in(1,2)$ and all integers $k\ge4$ and $r\ge1$, 
the growth rate $O_{\alpha,k,r}(x^{2-\alpha})$ is best in a certain meaning.
When $k=3$, we expect that $L_{\alpha,3,r}(x)=O_{\alpha,r}(x^{1-\alpha/2})$ for all $\alpha\in(1,2)$ and $r\in\mathbb{N}$.
For details, see Appendix~\ref{optimality}.

\section{Uniform distribution modulo $1$}\label{uni-dist}

To prove main theorems, 
uniform distribution modulo $1$ is a key point.
Unless there is confusion, $\{x\}$ denotes the fractional part $x-\lfloor{x}\rfloor$ of $x\in\mathbb{R}$.
Also, for $\mathbf{x}=(x_1,x_2,\ldots,x_d)\in \mathbb{R}^d$, 
define the notation 
\[
\{\mathbf{x}\} = (\{x_1\}, \{x_2\}, \ldots, \{x_d\}).
\]
Let $(\mathbf{x}_n)_{n=1}^\infty$ be a sequence of $\mathbb{R}^d$.
We say that $(\mathbf{x}_n)_{n=1}^\infty$ is \textit{uniformly distributed modulo $1$} 
if every convex set $\cC\subset[0,1)^d$ satisfies that 
\begin{equation}
	\lim_{N\to\infty}\frac{1}{N}
	\#\bigl\{ n\in[1,N]_{\mathbb{Z}} : \{\mathbf{x}_n\}\in\cC \bigr\}
	= \mu(\cC), \label{eq13}
\end{equation}
where $\mu$ denotes the Lebesgue measure on $\mathbb{R}^d$.
It is known that $(\mathbf{x}_n)_{n=1}^\infty$ is uniformly distributed modulo $1$ 
if and only if 
\begin{equation}\label{Weyl}
	\lim_{N\to\infty} \frac{1}{N}\sum_{n=1}^N e(\langle{\mathbf{h},\mathbf{x}_n}\rangle) = 0
\end{equation}
for all non-zero $\mathbf{h}\in\mathbb{Z}^d$, 
where the function $e(x)$ defined as $e^{2\pi ix}$, 
and $\langle{\cdot,\cdot}\rangle$ denotes the standard inner product on $\mathbb{R}^d$. 
One can also say that 
$(\mathbf{x}_n)_{n=1}^\infty$ is uniformly distributed modulo $1$ if and only if 
$(\langle{\mathbf{h},\mathbf{x}_n}\rangle)_{n=1}^\infty$ is uniformly distributed modulo $1$ for all non-zero $\mathbf{h}\in\mathbb{Z}^d$.
Due to this equivalence, the following facts hold: 
if a sequence $(\mathbf{x}_n)_{n=1}^\infty=\bigl( (x_{1,n},\ldots,x_{d,n}) \bigr)_{n=1}^\infty$ is uniformly distributed modulo $1$, 
then 
\begin{itemize}
	\item
	so is the sequence $(\mathbf{x}_n\mathbf{A})_{n=1}^\infty$ 
	for every integer matrix $\mathbf{A}$ of order $d$ and rank $d$;
	\item
	so is the sequence $(x_{i,n})_{n=1}^\infty$ for every $i\in[1,d]_{\mathbb{Z}}$.
\end{itemize}
For details, see \cite[Theorem~6.2]{KN}.

To investigate uniform distribution modulo $1$, 
we need to estimate exponential sums in general.
However, if a function $f\in\cH$ is \textit{subpolynomial}, i.e., $f(x)\ll x^n$ for some $n\in\mathbb{N}$, 
then it is easy to investigate whether 
the sequence $(f(n))_{n=n_0}^\infty$ is uniformly distributed modulo $1$.

\begin{proposition}[Boshernitzan \cite{Bo1}]\label{ud-Hardy}
	Let $n_0\in\mathbb{N}$.
	For every subpolynomial $f\in\cH$ defined on the interval $[n_0,\infty)$, 
	the following conditions are equivalent.
	\begin{itemize}
		\item
		$(f(n))_{n=n_0}^\infty$ is uniformly distributed modulo $1$.
		\item
		For every polynomial $p(x)\in\mathbb{Q}[x]$, 
		the ratio $(f(x)-p(x))/\log x$ diverges to positive or negative infinity as $x\to\infty$, 
		where the sign of infinity depends on $p$.
	\end{itemize}
\end{proposition}

The next corollary is a simple application of Proposition~\ref{ud-Hardy}.

\begin{corollary}
	Let $n_0=\lceil{e^e}\rceil=16$, 
	and let $f$ be a function in \eqref{equation=examples}.
	Then the sequence $\bigl( (f(n),f'(n)) \bigr)_{n=n_0}^\infty$ is uniformly distributed modulo $1$.
\end{corollary}
\begin{proof}
	Take a non-zero $(h_0,h_1)\in\mathbb{Z}^2$ arbitrarily.
	All we need is to show that 
	the sequence $(h_0f(n)+h_1f'(n))_{n=n_0}^\infty$ is uniformly distributed modulo $1$.
	It can be easily checked that for every $p(x)\in\mathbb{Q}[x]$ 
	the ratio $(h_0f(x)+h_1f'(x)-p(x))/\log x$ diverges to positive or negative infinity as $x\to\infty$.
	Since the function $h_0f+h_1f'$ belongs to $\cH$ and is subpolynomial, 
	Proposition~\ref{ud-Hardy} implies that the sequence $(h_0f(n)+h_1f'(n))_{n=n_0}^\infty$ is uniformly distributed modulo $1$.
	Therefore, we conclude this corollary.
\end{proof}

For the function $f(x)=x^\alpha$ with $\alpha\in(d,d+1)$ and $d\in\mathbb{N}$, 
it can be proved that 
the sequence $\bigl( (f(n), f'(n), f''(n)/2!, \ldots, f^{(d)}(n)/d!) \bigr)_{n=1}^\infty$ is uniformly distributed modulo $1$ 
in the same way as the above corollary.

Next, using uniform distribution modulo $1$, 
we state two propositions that imply Theorems~\ref{main0} and \ref{main1}.

\begin{proposition}\label{main0'}
	Let $n_0,d\in\mathbb{N}$, 
	and let $f\colon [n_0,\infty)\to\mathbb{R}$ be a $(d+1)$-times differentiable function 
	satisfying that 
	\begin{itemize}
		\item[(A1)]
		The $(d+1)$-st derivative $f^{(d+1)}(x)$ vanishes as $x\to\infty$;
		\item[(A2)]
		$\bigl( (f(n), f'(n), f''(n)/2!, \ldots, f^{(d)}(n)/d!) \bigr)_{n=1}^\infty$ is uniformly distributed modulo $1$;
		\item[(A3)]
		$\inf_{x\ge n_0} f'(x)\ge1$.
	\end{itemize}
	Then, for all integers $k\ge d+2$ and $r\ge1$, 
	the equality \eqref{eq:main0} holds.
	Also, $\mu(\cC_{k,d+1})$ is bounded below by $1/\prod_{i=1}^d \binom{k-1}{i}$.
\end{proposition}

\begin{proposition}\label{main1'}
	Let $n_0,d\in\mathbb{N}$, 
	and let $f\colon [n_0,\infty)\to\mathbb{R}$ be a $(d+1)$-times differentiable function 
	satisfying that 
	\begin{itemize}
		\item[(B1)]
		The $(d+1)$-st derivative $f^{(d+1)}(x)$ eventually decreases, 
		and vanishes as $x\to\infty$;
		\item[(B2)]
		$\lim_{x\to\infty} x^{d+1}f^{(d+1)}(x)=\infty$;
		\item[(B3)]
		For every $\delta\in(0,1)$, there exist $c(\delta)\ge1$ and $x_0(\delta)\ge n_0/\delta$ such that 
		every $x\ge x_0(\delta)$ satisfies $f^{(d+1)}(\delta x)\le c(\delta)f^{(d+1)}(x)$;
		\item[(B4)]
		$(f(n))_{n=n_0}^\infty$, $(f'(n))_{n=n_0}^\infty$, $(f''(n)/2!)_{n=n_0}^\infty$, ..., 
		$(f^{(d)}(n)/d!)_{n=n_0}^\infty$ are uniformly distributed modulo $1$;
		\item[(B5)]
		$\inf_{x\ge n_0} f'(x)\ge1$.
	\end{itemize}
	Then, for every integer $k\ge d+2$, the equality \eqref{eq:main1} holds.
\end{proposition}

The above propositions do not use the notion of Hardy fields, 
but uniform distribution modulo $1$ is used instead.
In general, it is not so easy to investigate uniform distribution modulo $1$, 
but it is easy for $f\in\cH$ as stated in Proposition~\ref{ud-Hardy}.
This is why we have used the notion of Hardy fields in Theorems~\ref{main0} and \ref{main1}.
Propositions~\ref{main0'} and \ref{main1'} are proved in Section~\ref{proofs}.

Before proving Theorems~\ref{main0} and \ref{main1} while assuming Propositions~\ref{main0'} and \ref{main1'}, 
we remark some properties of functions in $\cH$ \cite{FW}: 
\begin{itemize}
	\item[(H1)]
	Every $f\in\cH$ has eventually constant sign;
	\item[(H2)]
	Every $f\in\cH$ is eventually monotone;
	\item[(H3)]
	For every $f\in\cH$, 
	the limit $\displaystyle\lim_{x\to\infty} f(x)$ exists as an element of $\mathbb{R}\cup\{\pm\infty\}$;
	\item[(H4)]
	If $f\in\cH$ and if $g\in\cLE$ is eventually non-zero, then $f/g\in\cH$;
	\item[(H5)]
	For every $f\in\cH$ and every $g\in\cLE$ that is eventually non-zero, 
	the limit $\displaystyle\lim_{x\to\infty} f(x)/g(x)$ exists as an element of $\mathbb{R}\cup\{\pm\infty\}$;
	\item[(H6)]
	If eventually positive $f\in\cH$ and $g\in\cLE$ satisfy $f(x)\succ g(x)$ (resp.\ $f(x)\prec g(x)$) 
	and if $\displaystyle\lim_{x\to\infty} f(x)=\lim_{x\to\infty} g(x)=\infty$, 
	then $f'(x)\succ g'(x)$ (resp.\ $f'(x)\prec g'(x)$).
	\item[(H7)]
	If $f\in\cH$ is eventually positive, then $\log f(\cdot)\in\cH$.
\end{itemize}
Property (H1) is derived from the fact that $f\in\cH$ is eventually zero or has a reciprocal.
Property (H2) follows from (H1) by considering the derivative $f'\in\cH$.
Property (H3) follows from (H2) and the monotone convergence theorem.
Property (H5) follows from (H3) and (H4).
Property (H6) follows from (H5) and L'Hospital's rule.
For (H7), see \cite[Theorem~5.3]{Bo2}.
The remaining (H4) is verified as follows.
The set $\cLE$ is a Hardy field by the equivalence relation in Section~\ref{results} \cite{Hardy1,Hardy2}, 
and is contained in every maximal Hardy field 
(a Hardy field $\cF$ is called \textit{maximal} if there are not any Hardy fields strictly containing $\cF$) \cite{Bo2,Bo3}.
Also, for every Hardy field $\cF$, there exists a maximal Hardy field containing $\cF$ (use Zorn's lemma).
Therefore, for $f\in\cH$ and $g\in\cLE$ in (H4), the ratio $f/g$ belongs to $\cH$.

\begin{proof}[Proof of Theorems~$\ref{main0}$ and $\ref{main1}$ assuming  Propositions~$\ref{main0'}$ and $\ref{main1'}$]
	Let $f\colon [n_0,\infty)\to\mathbb{R}$ be a differentiable function in $\cH$ 
	and satisfy (a1) and (a2).
	All we need is to show (B1)--(B4) and (A2).
	\par
	Proof of (B1) and (B2).
	The relation $x^{-1}\prec f^{(d+1)}(x)\prec 1$ follows from (a1) and (H6).
	Thus, $f^{(d+1)}(x)$ converges to $+0$ as $x\to\infty$.
	This and (H2) imply (B1).
	Also, since the relation $f^{(d+1)}(x)\succ x^{-1}$ yields that 
	$xf^{(d+1)}(x)$ diverges to positive infinity as $x\to\infty$, 
	so does $x^{d+1}f^{(d+1)}(x)$.
	\par
	Proof of (A2) and (B4).
	Properties (a1) and (H6) imply that 
	$x^{d-i}\log x\prec f^{(i)}(x)\prec x^{d+1-i}$ for all $i\in[0,d]_{\mathbb{Z}}$.
	This fact and Proposition~\ref{ud-Hardy} imply (A2).
	Finally, (B4) follows from (A2) immediately.
	\par
	Proof of (B3).
	All we need is to show that for every $\delta\in(0,1)$, 
	\begin{equation}
		\limsup_{x\to\infty} \frac{f^{(d+1)}(\delta x)}{f^{(d+1)}(x)} < \infty.
		\label{eq03}
	\end{equation}
	Let $g(x)=1/f^{(d+1)}(x)$.
	Instead of \eqref{eq03}, we show that for every $\beta>1$, 
	\begin{equation}
		\limsup_{x\to\infty} \frac{g(\beta x)}{g(x)} < \infty,
		\label{eq03'}
	\end{equation}
	which is equivalent to \eqref{eq03}.
	First, the relation $1\prec g(x)\prec x$ follows from $x^{-1}\prec f^{(d+1)}(x)\prec 1$, 
	and moreover the function $\log g(\cdot)$ belongs to $\cH$ due to (H7).
	These facts and (H5) imply that 
	the ratio $\log g(x)/\log x$ converges to some finite $\gamma\in[0,1]$ as $x\to\infty$.
	Since both $\log g(x)$ and $\log x$ diverge to positive infinity, 
	L'Hospital's rule and (H5) yield that 
	\[
	\lim_{x\to\infty} \frac{xg'(x)}{g(x)}
	= \lim_{x\to\infty} \frac{g'(x)/g(x)}{1/x}
	= \lim_{x\to\infty} \frac{\log g(x)}{\log x} = \gamma.
	\]
	Thus, there exists $x_0>0$ such that $xg'(x)\le(\gamma+1)g(x)$ for all $x\ge x_0$.
	Also, since the relation $g'(x)\prec 1$ holds due to (H6), 
	the derivative $g'$ is eventually decreasing due to (H2).
	\par
	Let $\beta>1$.
	The mean value theorem implies that 
	$g(\beta x) - g(x) = (\beta-1)xg'(\beta'x)$ for some $\beta'=\beta'(x)\in(1,\beta)$.
	Since $g'$ is eventually decreasing, every sufficiently large $x\ge x_0$ satisfies 
	\[
	g(\beta x) - g(x) = (\beta-1)xg'(\beta'x)
	\le (\beta-1)xg'(x)
	\le (\beta-1)(\gamma+1)g(x).
	\]
	Therefore, the left-hand side in \eqref{eq03'} is bounded above by $(\beta-1)(\gamma+1)+1$.
\end{proof}

\section{Proofs of Propositions~\ref{main0'} and \ref{main1'}}\label{proofs}

First, we begin with the proof of Proposition~\ref{main0'}, 
which is a basis of subsequent proofs.

\begin{proof}[Proof of Proposition~$\ref{main0'}$]
	Without loss of generality, 
	we may assume $n_0=1$.
	Fix integers $k\ge d+2$ and $r,d\ge1$.
	Taylor's theorem implies that 
	for every $n\in\mathbb{N}$ and $j\in[1,k)_{\mathbb{Z}}$ 
	there exists $\theta=\theta(n,j)\in(n,n+rj)$ such that 
	\begin{equation}
		f(n+rj) = \sum_{l=0}^d \frac{(rj)^l}{l!}f^{(l)}(n) + \frac{(rj)^{d+1}}{(d+1)!}f^{(d+1)}(n+\theta).
		\label{eq05}
	\end{equation}
	The falling factorials satisfy the formula $x^n=\sum_{i=0}^n S(n,i)(x)_i$, 
	where $S(n,i)$, $i\in[0,n]_{\mathbb{Z}}$, denote the Stirling numbers of the second kind.
	Thus, \eqref{eq05} can be rewritten as 
	\[
	f(n+rj) = \sum_{i=0}^d a_i\binom{j}{i} + \frac{(rj)^{d+1}}{(d+1)!}f^{(d+1)}(n+\theta),
	\]
	where 
	\begin{equation}
		a_i=a_i(n) \coloneqq \sum_{l=i}^d \frac{r^l}{l!}f^{(l)}(n)S(l,i)i!.
		\label{eq17}
	\end{equation}
	For convenience, we set $s_0=0$ in this proof.
	For every $\mathbf{s}=(s_i)_{i=1}^d\in\mathbb{Z}^d$, $n\in\mathbb{N}$ and $j\in[1,k)_{\mathbb{Z}}$, 
	we have 
	\begin{equation}
		f(n+rj) = \sum_{i=0}^d (\lfloor{a_i}\rfloor - s_i)\binom{j}{i} + \delta_{\mathbf{s}},
		\label{eq06}
	\end{equation}
	where 
	\[
	\delta_{\mathbf{s}} = \delta_{\mathbf{s}}(n,j)
	\coloneqq \sum_{i=0}^d (\{a_i\} + s_i)\binom{j}{i} + \frac{(rj)^{d+1}}{(d+1)!}f^{(d+1)}(n+\theta).
	\]
	Let $\epsilon\in(0,1/2)$ be arbitrary.
	Thanks to (A1), 
	we can take $x_0\ge1$ such that every $x\ge x_0$ satisfies 
	\begin{equation*}
		\frac{(r(k-1))^{d+1}}{(d+1)!}\abs{f^{(d+1)}(x)} \le \epsilon.
	\end{equation*}
	\par
	Now, let us show that 
	\begin{equation}
		\liminf_{N\to\infty} \frac{1}{N}
		\#\{ n\in[1,N]_{\mathbb{Z}} : (\lfloor{f(n+rj)}\rfloor)_{j=0}^{k-1}\in\cP_{k,d} \}
		\ge \mu(\cC_{k,d+1}^{-}(\epsilon)), \label{eq07}
	\end{equation}
	where the convex set $\cC_{k,d+1}^{-}(\epsilon)$ is defined as 
	\begin{equation}
		\cC_{k,d+1}^{-}(\epsilon) = \biggl\{ (y_i)_{i=0}^d\in\mathbb{R}^{d+1} : 
		0\le y_0<1,\ \epsilon\le\sum_{i=0}^d \binom{j}{i}y_i<1-\epsilon\ (\forall j\in[1,k)_{\mathbb{Z}}) \biggr\}.
		\label{eqC-}
	\end{equation}
	If the relations $\mathbf{s}=(s_i)_{i=1}^d\in\mathbb{Z}^d$, $n\ge x_0$, 
	and $(\{a_i(n)\}+s_i)_{i=0}^d\in\cC_{k,d+1}^{-}(\epsilon)$ hold, 
	then $0\le\delta_{\mathbf{s}}(n,j)<1$ and 
	\[
	\lfloor{f(n+rj)}\rfloor = \sum_{i=0}^d (\lfloor{a_i(n)}\rfloor - s_i)\binom{j}{i}
	\]
	for all $j\in[1,k)_{\mathbb{Z}}$.
	This implies the inclusion relation 
	\begin{align}
		&\quad \bigcup_{s_1,\ldots,s_d\in\mathbb{Z}}
		\{ n\in[x_0,\infty)_{\mathbb{Z}} : (\{a_i(n)\}+s_i)_{i=0}^d\in\cC_{k,d+1}^{-}(\epsilon) \}
		\label{eq04}\\
		&\subset \{ n\in\mathbb{N} : (\lfloor{f(n+rj)}\rfloor)_{j=0}^{k-1}\in\cP_{k,d} \}\nonumber.
	\end{align}
	The union \eqref{eq04} is disjoint because 
	\begin{enumerate}
		\item
		the vectors $(\binom{j}{i})_{i=1}^d\in\mathbb{R}^d$, $j\in[1,k)_{\mathbb{Z}}$, span $\mathbb{R}^d$;
		\item
		thus, if $(s_i)_{i=1}^d,(s'_i)_{i=1}^d\in\mathbb{Z}^d$ are not equal to each other, 
		then $\sum_{i=1}^d \binom{j}{i}(s_i-s'_i)$ is a non-zero integer for some $j\in[1,k)_{\mathbb{Z}}$.
	\end{enumerate}
	Also, the vectors $\mathbf{a}(n) \coloneqq (a_0(n),a_1(n),\ldots,a_d(n))$, $n\in\mathbb{N}$, 
	can be expressed as 
	\begin{equation*}
		\mathbf{a}(n) = ( f(n), f'(n), f''(n)/2!, \ldots, f^{(d)}(n)/d! )\mathbf{A}
	\end{equation*}
	by using the integer matrix $\mathbf{A} = (a_{ij})_{0\le i,j\le d}$ 
	whose entry $a_{ij}$ is equal to $r^iS(i,j)j!$ if $i\ge j$, and zero if $i<j$.
	Note that $\mathbf{A}$ has full rank.
	Since $(\mathbf{a}(n))_{n=1}^\infty$ is uniformly distributed modulo $1$ thanks to (A2), 
	it turns out that 
	\begin{equation}
	\begin{split}
		&\quad \liminf_{N\to\infty} \frac{1}{N}
		\#\{ n\in[1,N]_{\mathbb{Z}} : (\lfloor{f(n+rj)}\rfloor)_{j=0}^{k-1}\in\cP_{k,d} \}\\
		&\ge \liminf_{N\to\infty} \sum_{s_1,\ldots,s_d\in\mathbb{Z}} \frac{1}{N}
		\#\{ n\in[x_0,N]_{\mathbb{Z}} : (\{a_i(n)\}+s_i)_{i=0}^d\in\cC_{k,d+1}^{-}(\epsilon) \}\\
		&= \sum_{s_1,\ldots,s_d\in\mathbb{Z}} \mu\Bigl( \cC_{k,d+1}^{-}(\epsilon)\cap\prod_{i=0}^d [s_i,s_i+1) \Bigr)
		= \mu(\cC_{k,d+1}^{-}(\epsilon)),
	\end{split}\label{sumC-}
	\end{equation}
	where all the sums in \eqref{sumC-} are finite sums because of the boundedness of $\cC_{k,d+1}^{-}(\epsilon)$.
	Therefore, \eqref{eq07} holds.
	\par
	Next, let us show that 
	\begin{equation}
		\limsup_{N\to\infty} \frac{1}{N}
		\#\{ n\in[1,N]_{\mathbb{Z}} : (\lfloor{f(n+rj)}\rfloor)_{j=0}^{k-1}\in\cP_{k,d} \}
		\le \mu(\cC_{k,d+1}^{+}(\epsilon)), \label{eq08}
	\end{equation}
	where the convex set $\cC_{k,d+1}^{+}(\epsilon)$ is defined as 
	\begin{equation}
		\cC_{k,d+1}^{+}(\epsilon) = \biggl\{ (y_i)_{i=0}^d\in\mathbb{R}^{d+1} : 
		0\le y_0<1,\ -\epsilon\le\sum_{i=0}^d \binom{j}{i}y_i<1+\epsilon\ (\forall j\in[1,k)_{\mathbb{Z}}) \biggr\}.
		\label{eqC+}
	\end{equation}
	Take an arbitrary integer $m\ge x_0$ such that $(\Delta_r^d\lfloor{f(m+rj)}\rfloor)_{j=0}^{k-d-1}$ is a constant sequence.
	Then the sequence $(\lfloor{f(m+rj)}\rfloor)_{j=0}^{k-1}$ is expressed as 
	\[
	\lfloor{f(m+rj)}\rfloor = \sum_{i=0}^d \Delta_r^i\lfloor{f(m)}\rfloor\cdot\binom{j}{i}
	\]
	due to Newton's forward difference formula.
	Recalling the definition of $a_i(m)$ and 
	putting $s_i=\lfloor{a_i(m)}\rfloor-\Delta_r^{i}\lfloor{f(m)}\rfloor$ for $i\in[1,d]_{\mathbb{Z}}$, 
	we have that $\lfloor{a_0(m)}\rfloor=\lfloor{f(m)}\rfloor$ and 
	\[
	\lfloor{f(m+rj)}\rfloor = \sum_{i=0}^d (\lfloor{a_i(m)}\rfloor-s_i)\binom{j}{i}
	\quad(\forall j\in[0,k)_{\mathbb{Z}}).
	\]
	This and \eqref{eq06} imply that 
	$f(m+rj) = \lfloor{f(m+rj)}\rfloor + \delta_{\mathbf{s}}(m,j)$ 
	and $0\le\delta_{\mathbf{s}}(m,j)<1$ for all $j\in[1,k)_{\mathbb{Z}}$, 
	whence $(\{a_i(m)\}+s_i)_{i=0}^d\in\cC_{k,d+1}^{+}(\epsilon)$.
	Therefore, we obtain the inclusion relation 
	\begin{align*}
		&\quad \{ n\in[x_0,\infty)_{\mathbb{Z}} : (\lfloor{f(n+rj)}\rfloor)_{j=0}^{k-1}\in\cP_{k,d} \}\\
		&\subset \bigcup_{s_1,\ldots,s_d\in\mathbb{Z}}
		\{ n\in\mathbb{N} : (\{a_i(n)\}+s_i)_{i=0}^d\in\cC_{k,d+1}^{+}(\epsilon) \}.
	\end{align*}
	Since $(\mathbf{a}(n))_{n=1}^\infty$ is uniformly distributed modulo $1$ thanks to (A2), 
	it turns out that 
	\begin{equation}
	\begin{split}
		&\quad \limsup_{N\to\infty} \frac{1}{N}
		\#\{ n\in[1,N]_{\mathbb{Z}} : (\lfloor{f(n+rj)}\rfloor)_{j=0}^{k-1}\in\cP_{k,d} \}\\
		&= \limsup_{N\to\infty} \frac{1}{N}
		\#\{ n\in[x_0,N]_{\mathbb{Z}} : (\lfloor{f(n+rj)}\rfloor)_{j=0}^{k-1}\in\cP_{k,d} \}\\
		&\le \limsup_{N\to\infty} \sum_{s_1,\ldots,s_d\in\mathbb{Z}} \frac{1}{N}
		\#\{ n\in[1,N]_{\mathbb{Z}} : (\{a_i(n)\}+s_i)_{i=0}^d\in\cC_{k,d+1}^{+}(\epsilon) \}\\
		&= \sum_{s_1,\ldots,s_d\in\mathbb{Z}} \mu\Bigl( \cC_{k,d+1}^{+}(\epsilon)\cap\prod_{i=0}^d [s_i,s_i+1) \Bigr)
		= \mu(\cC_{k,d+1}^{+}(\epsilon)),
	\end{split}\label{sumC+}
	\end{equation}
	which is just \eqref{eq08}.
	\par
	Finally, once letting $\epsilon\to+0$ in \eqref{eq07} and \eqref{eq08}, 
	we conclude the limit in Proposition~\ref{main0'}.
	Also, the inequality $\mu(\cC_{k,d+1})\ge1/\prod_{i=1}^d \binom{k-1}{i}$ 
	is derived from the lemma below.
\end{proof}

\begin{lemma}\label{conv-set-lb}
	Let $k\ge d+2$ and $d\ge1$ be integers.
	Then $\mu(\cC_{k,d+1})\ge1/\prod_{i=1}^d \binom{k-1}{i}$.
\end{lemma}
\begin{proof}
	Define the convex set $\cC'_{k,d+1}$ as 
	\[
	\cC'_{k,d+1} =
	\biggl\{ (y_0,y_1,\ldots,y_d)\in\mathbb{R}^{d+1} : 0\le y_0<1,\ 0\le \sum_{i=0}^j \binom{k-1}{i}y_i < 1\ (\forall j\in[1,d]_{\mathbb{Z}}) \biggr\}.
	\]
	We show the inclusion relation $\cC'_{k,d+1}\subset\cC_{k,d+1}$.
	Let $(y_0,y_1,\ldots,y_d)\in\cC'_{k,d+1}$ and $j\in[1,k)_{\mathbb{Z}}$.
	Set the real numbers $c_0,c_1,\ldots,c_d\ge0$ as 
	\[
	c_l =
	\begin{dcases}
		\binom{j}{l}\binom{k-1}{l}^{-1} - \binom{j}{l+1}\binom{k-1}{l+1}^{-1} & l\in[0,d)_{\mathbb{Z}},\\
		\binom{j}{d}\binom{k-1}{d}^{-1} & l=d.
	\end{dcases}
	\]
	Then the inequality $0\le\sum_{i=0}^d \binom{j}{i}y_i<1$ in the definition of $\cC_{k,d+1}$ is equal to 
	the sum of the inequalities $0\le\sum_{i=0}^l \binom{k-1}{i}y_i<1$, $l\in[0,d]_{\mathbb{Z}}$, multiplied by $c_l$: 
	\begin{align*}
		&\quad \sum_{l=0}^d c_l\sum_{i=0}^l \binom{k-1}{i}y_i
		= \sum_{i=0}^d \binom{k-1}{i}y_i\sum_{l=i}^d c_l\\
		&= \sum_{i=0}^d \binom{k-1}{i}y_i\binom{j}{i}\binom{k-1}{i}^{-1}
		= \sum_{i=0}^d \binom{j}{i}y_i.
	\end{align*}
	Since $j\in[1,k)_{\mathbb{Z}}$ is arbitrary, 
	the point $(y_0,y_1,\ldots,y_d)$ lies in $\cC_{k,d+1}$.
	Therefore, $\cC'_{k,d+1}\subset\cC_{k,d+1}$.
	Finally, we conclude that $\mu(\cC_{k,d+1})\ge\mu(\cC'_{k,d+1})=1/\prod_{i=1}^d \binom{k-1}{i}$ by easy calculation.
\end{proof}

\begin{remark}
	Let $f(x)=x\log x$.
	Then the sequence $\bigl( (f(n),f'(n)) \bigr)_{n=1}^\infty$ is not uniformly distributed modulo $1$ 
	because $(f'(n))_{n=1}^\infty$ does not satisfy the second condition in Proposition~\ref{ud-Hardy}.
	However, one can show that for every convex set $\cC\subset[0,1)^2$ and every $r\in\mathbb{N}$, 
	\begin{align*}
		&\quad \frac{1}{N}\#\{ n\in[1,N]_{\mathbb{Z}} : (\{f(n)\},\{rf'(n)\})\in\cC \}\\
		&= \iint_\cC \Bigl( \mathbf{1}_{\le\{r\log N\}}(y) + \frac{1}{e^{1/r}-1} \Bigr)\frac{e^{(y-\{r\log N\})/r}}{r}\,dxdy
		\quad(N\to\infty),
	\end{align*}
	where $\mathbf{1}_{\le c}(y)=1$ if $y\le c$, 
	and $\mathbf{1}_{\le c}(y)=0$ if $y>c$.
	This implies that for every convex set $\cC\subset[0,1)^2$ and every $r\in\mathbb{N}$, 
	\begin{align*}
		&\quad \frac{1}{(e^{1/r}-1)r}\mu(\cC)
		\le \liminf_{N\to\infty} \frac{1}{N}\#\{ n\in[1,N]_{\mathbb{Z}} : (\{f(n)\},\{rf'(n)\})\in\cC \}\\
		&\le \limsup_{N\to\infty} \frac{1}{N}\#\{ n\in[1,N]_{\mathbb{Z}} : (\{f(n)\},\{rf'(n)\})\in\cC \}
		\le \frac{e^{1/r}}{(e^{1/r}-1)r}\mu(\cC).
	\end{align*}
	Hence, it follows that for all integers $k\ge3$ and $r\ge1$, 
	\begin{align*}
		&\quad \frac{1}{(e^{1/r}-1)r(k-1)}
		\le \liminf_{N\to\infty} \frac{1}{N}
		\#\{ n\in[1,N]_{\mathbb{Z}} : (\lfloor{f(n+rj)}\rfloor)_{j=0}^{k-1}\in\cP_{k,1} \}\\
		&\le \limsup_{N\to\infty} \frac{1}{N}
		\#\{ n\in[1,N]_{\mathbb{Z}} : (\lfloor{f(n+rj)}\rfloor)_{j=0}^{k-1}\in\cP_{k,1} \}
		\le \frac{e^{1/r}}{(e^{1/r}-1)r(k-1)}
	\end{align*}
	in the same way as the proof of Proposition~\ref{main0'}.
	The above both-hand sides converge to $1/(k-1)$ as $r\to\infty$.
\end{remark}

Next, to prove Proposition~\ref{main1'}, 
we need to evaluate exponential sums $\sum_{r=1}^R e(p(r))$ for polynomials $p(x)$.
Such an evaluation is achieved by induction on the degree of $p(x)$.
The following lemma is often used to make the degree of a polynomial decrease.

\begin{lemma}\label{diff-ineq}
	Let $z_1,z_2,\ldots,z_N\in\mathbb{C}$ and $H\in[1,N]_{\mathbb{Z}}$.
	Then 
	\[
	\abs{\sum_{n=1}^N z_n}^2
	\le \frac{N+H-1}{H^2}\Bigl( H\sum_{n=1}^N |z_n|^2
	+ 2\sum_{h=1}^{H-1} (H-h)\Re\sum_{n=1}^{N-h} z_{n+h}\overline{z}_n \Bigr).
	\]
\end{lemma}
\begin{proof}
	See \cite[Lemma~3.1]{KN}.
\end{proof}

\begin{lemma}\label{lem:exp-sum}
	Let $N_m,R_m\in\mathbb{N}$ diverge to positive infinity as $m\to\infty$, 
	and $d\ge0$ be an integer.
	For $n\in\mathbb{N}$, let $q_n(x)$ be a polynomial of degree less than $d$; 
	let $c_n\in\mathbb{R}$ and $p_n(x) = c_nx^d + q_n(x)$.
	If $(c_n)_{n=1}^\infty$ is uniformly distributed modulo $1$, 
	then 
	\[
	\lim_{m\to\infty} \frac{1}{N_mR_m}\sum_{n=1}^{N_m} \sum_{r=1}^{R_m} e(p_n(r)) = 0.
	\]
\end{lemma}
\begin{proof}
	We show the desired statement by induction on $d$.
	First, assume $d=0$.
	Then $p_n(x)=c_n$ for all $n\in\mathbb{N}$, and thus 
	the uniform distribution modulo $1$ of $(c_n)_{n=1}^\infty$ implies that 
	\[
	\lim_{m\to\infty} \frac{1}{N_mR_m}\sum_{n=1}^{N_m} \sum_{r=1}^{R_m} e(p_n(r))
	= \lim_{m\to\infty} \frac{1}{N_m}\sum_{n=1}^{N_m} e(c_n) = 0.
	\]
	\par
	Next, assuming that the desired statement is true for $d-1$ with $d\ge1$, 
	we show that the desired statement is also true for $d$.
	Take an arbitrary $H\in\mathbb{N}$.
	Lemma~\ref{diff-ineq} yields the inequality 
	\[
	\abs{\sum_{r=1}^{R_m} e(p_n(r))}^2
	\le \frac{R_m+H-1}{H^2}\Bigl( HR_m + 2\sum_{h=1}^{H-1} (H-h)\Re\sum_{r=1}^{R_m-h} e(\Delta_h p_n(r)) \Bigr).
	\]
	The above and Cauchy-Schwarz inequalities imply that 
	\begin{equation}
	\begin{split}
		&\quad \abs{\frac{1}{N_mR_m}\sum_{n=1}^{N_m} \sum_{r=1}^{R_m} e(p_n(r))}^2
		\le \frac{1}{N_mR_m^2}\sum_{n=1}^{N_m} \abs{\sum_{r=1}^{R_m} e(p_n(r))}^2\\
		&\le \frac{R_m+H-1}{H^2N_mR_m^2}\Biggl( HN_mR_m + 2\sum_{h=1}^{H-1} (H-h)\Re\sum_{n=1}^{N_m} \sum_{r=1}^{R_m-h} e(\Delta_h p_n(r)) \Biggr)\\
		&\le \frac{R_m+H}{H^2N_mR_m^2}\Biggl( HN_mR_m + 2H\sum_{h=1}^{H-1} \abs{\sum_{n=1}^{N_m} \sum_{r=1}^{R_m-h} e(\Delta_h p_n(r))} \Biggr)\\
		&\le \frac{R_m+H}{HR_m}\Biggl( 1 + 2\sum_{h=1}^{H-1} \abs{\frac{1}{N_mR_m}\sum_{n=1}^{N_m} \sum_{r=1}^{R_m-h} e(\Delta_h p_n(r))} \Biggr).
	\end{split}\label{eq19}
	\end{equation}
	Now, for all $h,n\in\mathbb{N}$, 
	the polynomial $\Delta_h p_n(x)$ is expressed as $dhc_nx^{d-1} + q_{h,n}(x)$, 
	where the degree of $q_{h,n}(x)$ is less than $d-1$.
	Since $(c_n)_{n=1}^\infty$ is uniformly distributed modulo $1$, 
	so is $(dhc_n)_{n=1}^\infty$ for every $h\in\mathbb{N}$.
	Thus, the hypothesis by induction implies that for every $h\in\mathbb{N}$ 
	\[
	\lim_{m\to\infty} \frac{1}{N_mR_m}\sum_{n=1}^{N_m} \sum_{r=1}^{R_m-h} e(\Delta_h p_n(r)) = 0.
	\]
	It follows from \eqref{eq19} that 
	\[
	\limsup_{m\to\infty} \abs{\frac{1}{N_mR_m}\sum_{n=1}^{N_m} \sum_{r=1}^{R_m} e(p_n(r))}^2
	\le 1/H.
	\]
	Due to the arbitrariness of $H$, 
	we find that the desired statement is true for $d$.
\end{proof}

\begin{lemma}\label{lem:Weyl}
	Let $d\ge0$ be an integer and $\mathbf{A}$ be an integer matrix of order $d+1$ and rank $d+1$; 
	let $\mathbf{x}(n)=(x_0(n),x_1(n),\ldots,x_d(n))\in\mathbb{R}^{d+1}$ and 
	\[
	\mathbf{y}(n,r) = (y_0(n,r),y_1(n,r),\ldots,y_d(n,r))
	= (x_0(n),rx_1(n),\ldots,r^dx_d(n))\mathbf{A}
	\]
	for $n,r\in\mathbb{N}$.
	If $N_m,R_m\in\mathbb{N}$ diverge to positive infinity as $m\to\infty$ 
	and if each entry $(x_i(n))_{n=1}^\infty$ of $(\mathbf{x}(n))_{n=1}^\infty$ is uniformly distributed modulo $1$, 
	then for every convex set $\cC\subset[0,1)^{d+1}$, 
	\begin{equation*}
		\lim_{m\to\infty} \frac{\#\bigl\{ (n,r)\in[1,N_m]_{\mathbb{Z}}\times[1,R_m]_{\mathbb{Z}} : \{\mathbf{y}(n,r)\}\in\cC \bigr\}}{N_mR_m}
		= \mu(\cC),
	\end{equation*}
	where $\mu$ denotes the Lebesgue measure on $\mathbb{R}^{d+1}$.
\end{lemma}
\begin{proof}
	If the following criterion holds, 
	Lemma~\ref{lem:Weyl} follows in the same way as Weyl's theorem on uniform distribution.
	\textit{Weyl's criterion: for every non-zero $\mathbf{h}=(h_0,h_1,\ldots,h_d)\in\mathbb{Z}^{d+1}$,} 
	\begin{equation}
		\lim_{m\to\infty} \frac{1}{N_mR_m}\sum_{n=1}^{N_m} \sum_{r=1}^{R_m} e(\langle{\mathbf{y}(n,r),\mathbf{h}}\rangle)
		= 0. \label{eq:Weyl}
	\end{equation}
	Hence, taking a non-zero $\mathbf{h}\in\mathbb{Z}^{d+1}$ arbitrarily, 
	we show \eqref{eq:Weyl}.
	For $i,j\in[0,d]_{\mathbb{Z}}$, denote the $(i,j)$-th entry of $\mathbf{A}$ by $a_{ij}$.
	For $n\in\mathbb{N}$, regard $\langle{\mathbf{y}(n,r),\mathbf{h}}\rangle$ as a polynomial $p_n(r)$ of $r$: 
	\[
	p_n(r) = \langle{\mathbf{y}(n,r),\mathbf{h}}\rangle
	= \sum_{j=0}^d y_j(n,r)h_j
	= \sum_{i=0}^d r^ix_i(n)\sum_{j=0}^d a_{ij}h_j.
	\]
	Take the maximum number $i_0$ of all $i\in[0,d]_{\mathbb{Z}}$ 
	such that $\sum_{j=0}^d a_{ij}h_j$ is not zero 
	(such a number $i$ exists because the square matrix $\mathbf{A}$ has full rank).
	Then, for every $n\in\mathbb{N}$, the degree of $p_n(x)$ is at most $i_0$.
	Since $(x_{i_0}(n))_{n=1}^\infty$ is uniformly distributed modulo $1$, 
	so is the sequence $(x_{i_0}(n)\sum_{j=0}^d a_{i_0j}h_j)_{n=1}^\infty$.
	Therefore, Lemma~\ref{lem:exp-sum} implies \eqref{eq:Weyl}, 
	and we obtain Lemma~\ref{lem:Weyl}.
\end{proof}

Now, let us show Proposition~\ref{main1'}.
Since \eqref{eq:main1} consists of the following inequalities: 
\begin{align*}
	&\text{(liminf)} &
	&\liminf_{N\to\infty} \frac{\#\{ P\subset[1,N]_{\mathbb{Z}} : 
	P\in\cP_{k,1},\ (\lfloor{f(n)}\rfloor)_{n\in P}\in\cP_{k,d} \}}{Nf''(N)^{-1/(d+1)}}
	> 0,\\
	&\text{(limsup)} &
	&\limsup_{N\to\infty} \frac{\#\{ P\subset[1,N]_{\mathbb{Z}} : 
	P\in\cP_{k,1},\ (\lfloor{f(n)}\rfloor)_{n\in P}\in\cP_{k,d} \}}{Nf''(N)^{-1/(d+1)}}
	< \infty,
\end{align*}
we prove the above inequalities.
Also, note that $f^{(d+1)}(x)>0$ for every sufficiently large $x>0$ because of (B2).

\begin{proof}[Proof of Proposition~$\ref{main1'}$ (liminf)]
	Without loss of generality, 
	we may assume $n_0=1$.
	Fix integers $k\ge d+2$ and $d\ge1$, 
	and let $N\in\mathbb{N}$ be sufficiently large.
	Take arbitrary $\epsilon\in(0,1)$ and $0<\delta_1<\cdots<\delta_t<\delta_{t+1}=1$.
	Put 
	\[
	R_i = R_i(N) = \Bigl\lfloor{\frac{(\epsilon(d+1)!)^{1/(d+1)}}{k-1}f^{(d+1)}(\delta_i N)^{-1/(d+1)}}\Bigr\rfloor
	\]
	for $i\in[t]$ and $N\in\mathbb{N}$.
	Then all $x\ge\delta_i N$ and $r\in[1,R_i]_{\mathbb{Z}}$ satisfy 
	\begin{equation}
		0 < \frac{(r(k-1))^{d+1}}{(d+1)!}f^{(d+1)}(x)
		\overset{\text{(B1)}}{\le} \frac{(R_i(k-1))^{d+1}}{(d+1)!}f^{(d+1)}(\delta_i N)
		\le \epsilon.
		\label{Y3}
	\end{equation}
	Now, the following inequality holds: 
	\begin{equation}
	\begin{split}
		&\quad \#\{ P\subset[1,N]_{\mathbb{Z}} : P\in\cP_{k,1},\ (\lfloor{f(n)}\rfloor)_{n\in P}\in\cP_{k,d} \}\\
		&\ge \#\{ (n,r)\in[1,N-(k-1)R_t]_{\mathbb{Z}}\times[1,R_t]_{\mathbb{Z}} : (\lfloor{f(n+rj)}\rfloor)_{j=0}^{k-1}\in\cP_{k,d} \}\\
		&\ge \#\{ (n,r)\in[1,N]_{\mathbb{Z}}\times[1,R_t]_{\mathbb{Z}} : (\lfloor{f(n+rj)}\rfloor)_{j=0}^{k-1}\in\cP_{k,d} \}
		- (k-1)R_t^2\\
		&\overset{(a)}{\ge} \sum_{i=1}^t M_i(N) - (k-1)R_t^2,
	\end{split}\label{Y4}
	\end{equation}
	where for $i\in[1,t]_{\mathbb{Z}}$ and $N\in\mathbb{N}$ the value $M_i(N)$ is defined as 
	\[
	M_i(N) =
	\#\{ (n,r)\in(\delta_i N,\delta_{i+1}N]_{\mathbb{Z}}\times[1,R_i]_{\mathbb{Z}} : (\lfloor{f(n+rj)}\rfloor)_{j=0}^{k-1}\in\cP_{k,d} \},
	\]
	and the monotonicity $R_1\le R_2\le\cdots\le R_t$ is used to obtain $(a)$.
	For $n,r\in\mathbb{N}$ and $i\in[0,d]_{\mathbb{Z}}$, 
	define the real number $a_i=a_i(n,r)$ as the right-hand side in \eqref{eq17}.
	Then the vectors $\mathbf{a}(n,r) \coloneqq (a_0(n,r),a_1(n,r),\ldots,a_d(n,r))$, $n,r\in\mathbb{N}$, 
	can be expressed as 
	\[
	\mathbf{a}(n,r) = ( f(n), rf'(n), r^2f''(n)/2!, \ldots, r^df^{(d)}(n)/d! )\mathbf{A},
	\]
	where the integer matrix $\mathbf{A}=(a_{ij})_{0\le i,j\le d}$ is defined as 
	$a_{ij}=S(i,j)j!$ if $i\ge j$, and $a_{ij}=0$ if $i<j$.
	Also, define the convex set $\cC_{k,d+1}^{-}(\epsilon)$ as 
	\[
	\cC_{k,d+1}^{-}(\epsilon) = \biggl\{ (y_i)_{i=0}^d\in\mathbb{R}^{d+1} : 
	0\le y_0<1,\ 0\le\sum_{i=0}^d \binom{j}{i}y_i<1-\epsilon\ (\forall j\in[1,k)_{\mathbb{Z}}) \biggr\}.
	\]
	Due to \eqref{Y3}, the same argument as the proof of Proposition~\ref{main0'} implies that 
	if integers $n\ge\delta_iN$ and $r\in[1,R_i]_{\mathbb{Z}}$ and a vector $(s_j)_{j=1}^d\in\mathbb{Z}^d$ 
	satisfy $(\{a_j(n,r)\}+s_j)_{j=0}^d\in\cC_{k,d+1}^{-}(\epsilon)$, 
	then $(\lfloor{f(n+rj)}\rfloor)_{j=0}^{k-1}$ belongs to $\cP_{k,d}$, 
	where $s_0\coloneqq0$.
	Thus, 
	\begin{align}
		&\quad \frac{\#\{ P\subset[1,N]_{\mathbb{Z}} : P\in\cP_{k,1},\ (\lfloor{f(n)}\rfloor)_{n\in P}\in\cP_{k,d} \}}{Nf^{(d+1)}(N)^{-1/(d+1)}}\nonumber\\
		&\overset{\eqref{Y4}}{\ge} \sum_{i=1}^t
		\frac{M_i(N)}{Nf^{(d+1)}(N)^{-1/(d+1)}} - \frac{(k-1)R_t^2}{Nf^{(d+1)}(N)^{-1/(d+1)}}\nonumber\\
		&\ge \sum_{i=1}^t \sum_{s_1,\ldots,s_d\in\mathbb{Z}}
		\frac{M'_i(\delta_i,\delta_{i+1},N)}{Nf^{(d+1)}(N)^{-1/(d+1)}} - \frac{(k-1)R_t^2}{Nf^{(d+1)}(N)^{-1/(d+1)}}, \label{eq09}
	\end{align}
	where for $i\in[1,t]_{\mathbb{Z}}$, $N\in\mathbb{N}$ and $y>x\ge0$ the value $M'_i(x,y,N)$ is defined as 
	\[
	M'_i(x,y,N) =
	\#\{ (n,r)\in(xN,yN]_{\mathbb{Z}}\times[1,R_i]_{\mathbb{Z}} : (\{a_j(n,r)\}+s_j)_{j=0}^d\in\cC_{k,d+1}^{-}(\epsilon) \}.
	\]
	The absolute value of the second term of \eqref{eq09} is bounded above by 
	\begin{align*}
		&\quad \frac{(k-1)R_t^2}{Nf^{(d+1)}(N)^{-1/(d+1)}}
		\le \frac{k-1}{Nf^{(d+1)}(N)^{-1/(d+1)}}\cdot\frac{(\epsilon(d+1)!)^{2/(d+1)}}{(k-1)^2}f^{(d+1)}(\delta_tN)^{-2/(d+1)}\\
		&\overset{\text{(B1)}}{\le} \frac{f^{(d+1)}(N)^{-2/(d+1)}}{Nf^{(d+1)}(N)^{-1/(d+1)}}\cdot\frac{(\epsilon(d+1)!)^{2/(d+1)}}{k-1}
		\le \frac{1}{Nf^{(d+1)}(N)^{1/(d+1)}}\cdot\frac{(d+1)!}{k-1} \underset{\text{(B2)}}{\xrightarrow{N\to\infty}} 0.
	\end{align*}
	Also, the following inequality holds: 
	\begin{align*}
		&\quad \frac{M'_i(\delta_i,\delta_{i+1},N)}{Nf^{(d+1)}(N)^{-1/(d+1)}}
		\overset{\text{(B3)}}{\ge} \frac{M'_i(\delta_i,\delta_{i+1},N)}{c(\delta_i)^{1/(d+1)}Nf^{(d+1)}(\delta_i N)^{-1/(d+1)}}\\
		&= \frac{M'_i(0,\delta_{i+1},N) - M'_i(0,\delta_i,N)}{c(\delta_i)^{1/(d+1)}Nf^{(d+1)}(\delta_i N)^{-1/(d+1)}}\\
		&= \frac{M'_i(0,\delta_{i+1},N)}{\delta_{i+1}NR_i}
		\cdot\frac{c(\delta_i)^{-1/(d+1)}\delta_{i+1}R_i}{f^{(d+1)}(\delta_i N)^{-1/(d+1)}}
		- \frac{M'_i(0,\delta_i,N)}{\delta_i NR_i}
		\cdot\frac{c(\delta_i)^{-1/(d+1)}\delta_i R_i}{f^{(d+1)}(\delta_i N)^{-1/(d+1)}}.
	\end{align*}
	Once taking the limit $N\to\infty$ in the above inequality, 
	Lemma~\ref{lem:Weyl} implies that 
	\begin{align*}
		&\quad \liminf_{N\to\infty} \frac{M'_i(\delta_i,\delta_{i+1},N)}{Nf^{(d+1)}(N)^{-1/(d+1)}}\\
		&\ge \mu\Bigl( \cC_{k,d+1}^{-}(\epsilon)\cap\prod_{j=0}^d [s_j,s_j+1) \Bigr)
		\cdot \frac{(\epsilon(d+1)!)^{1/(d+1)}c(\delta_i)^{-1/(d+1)}}{k-1}(\delta_{i+1}-\delta_i).
	\end{align*}
	Therefore, letting $N\to\infty$ in \eqref{eq09}, we obtain 
	\begin{align*}
		&\quad \liminf_{N\to\infty} \frac{\#\{ P\subset[1,N]_{\mathbb{Z}} : P\in\cP_{k,1},\ (\lfloor{f(n)}\rfloor)_{n\in P}\in\cP_{k,d} \}}{Nf^{(d+1)}(N)^{-1/(d+1)}}\\
		&\ge \sum_{i=1}^t \sum_{s_1,\ldots,s_d\in\mathbb{Z}}
		\mu\Bigl( \cC_{k,d+1}^{-}(\epsilon)\cap\prod_{j=0}^d [s_j,s_j+1) \Bigr)
		\cdot \frac{(\epsilon(d+1)!)^{1/(d+1)}c(\delta_i)^{-1/(d+1)}}{k-1}(\delta_{i+1}-\delta_i)\\
		&= \mu(\cC_{k,d+1}^{-}(\epsilon))(\epsilon(d+1)!)^{1/(d+1)}\sum_{i=1}^t \frac{c(\delta_i)^{-1/(d+1)}}{k-1}(\delta_{i+1}-\delta_i)
		> 0,
	\end{align*}
	where the last inequality is derived from 
	$\mu(\cC_{k,d+1}^{-}(\epsilon))\ge(1-\epsilon)^{d+1}/\prod_{i=1}^d \binom{k-1}{i} > 0$ 
	(see Lemma~\ref{conv-set-lb}).
\end{proof}

\begin{remark}\label{rem:lb}
	Let us consider the special case $f(x)=x^\alpha$ with $\alpha\in(d,d+1)$.
	Then we can take $c(\delta)$ in (B3) as $c(\delta)=\delta^{\alpha-d-1}$.
	Thus, 
	\begin{align*}
		&\quad \liminf_{N\to\infty}
		\frac{\#\{ P\subset[1,N]_{\mathbb{Z}} : P\in\cP_{k,1},\ (\lfloor{n^\alpha}\rfloor)_{n\in P}\in\cP_{k,d} \}}{((\alpha)_{d+1})^{-1/(d+1)}N^{2-\alpha/(d+1)}}\\
		&\ge \mu(\cC_{k,d+1}^{-}(\epsilon))(\epsilon(d+1)!)^{1/(d+1)}
		\sum_{i=1}^t \delta_i^{1-\alpha/(d+1)}(\delta_{i+1}-\delta_i).
	\end{align*}
	Since $\epsilon\in(0,1)$ and $0<\delta_1<\cdots<\delta_t<\delta_{t+1}=1$ are arbitrary, 
	we obtain 
	\begin{align*}
		&\quad \liminf_{N\to\infty} \frac{\#\{ P\subset[1,N]_{\mathbb{Z}} : P\in\cP_{k,1},\ (\lfloor{n^\alpha}\rfloor)_{n\in P}\in\cP_{k,d} \}}{N^{2-\alpha/(d+1)}}\\
		&\ge C_{k,d}\Bigl( \frac{(d+1)!}{(\alpha)_{d+1}} \Bigr)^{1/(d+1)} \int_0^1 x^{1-\alpha/(d+1)}\,dx\\
		&= C_{k,d}\Bigl( \frac{(d+1)!}{(\alpha)_{d+1}} \Bigr)^{1/(d+1)}\frac{1}{2-\alpha/(d+1)} \eqqcolon \tilde{A}_{\alpha,k},
	\end{align*}
	where 
	\[
	C_{k,d} = \sup_{0<x<1} \mu(\cC_{k,d+1}^{-}(x))x^{1/(d+1)}
	\ge \frac{\sup_{0<x<1} (1-x)^{d+1} x^{1/(d+1)}}{\prod_{i=1}^d \binom{k-1}{i}}
	\]
	because the inequality $\mu(\cC_{k,d+1}^{-}(x)) \ge (1-x)^{d+1}/\prod_{i=1}^d \binom{k-1}{i}$
	is derived from Lemma~\ref{conv-set-lb}.
	Therefore, the constant $A_{\alpha,k}$ in Theorem~\ref{main01} 
	is an arbitrary value in the interval $(0,\tilde{A}_{\alpha,k})$.
\end{remark}

\begin{proof}[Proof of Proposition~$\ref{main1'}$ (limsup)]
	Without loss of generality, 
	we may assume $n_0=1$.
	Fix integers $k\ge d+2$ and $d\ge1$, 
	and take an arbitrary $\beta>1$.
	Due to (B2), we can take an integer $N_0\ge x_0(1/\beta)$ such that 
	every $x\ge N_0$ satisfies that $f^{(d+1)}(x)>0$ and $1+(k-1)R(x)/x<\beta$, 
	where 
	\begin{equation*}
		R(x) \coloneqq \Bigl( \frac{2^d c(1/\beta)}{k-d-1} \Bigr)^{1/(d+1)}f^{(d+1)}(x)^{-1/(d+1)}.
	\end{equation*}
	First, we show that 
	if integers $n\ge N_0$ and $r\ge1$ satisfy $(\lfloor{f(n+rj)}\rfloor)_{j=0}^{k-1}\in\cP_{k,d}$, 
	then $r < R(n)$ by contradiction.
	Suppose that integers $m_0\ge N_0$ and $r_0\ge1$ satisfied that 
	$(\lfloor{f(m_0+r_0j)}\rfloor)_{j=0}^{k-1}\in\cP_{k,d}$ and $r_0 \ge R_0 \coloneqq R(m_0)$.
	The derivative of the function 
	\begin{align*}
		g(x) &\coloneqq \Delta_x^d f(m_0+(k-d-1)x) - \Delta_x^d f(m_0)\\
		&= \sum_{i=0}^d \binom{d}{i}(-1)^i f(m_0+(k-1-i)x) - \sum_{i=0}^d \binom{d}{i}(-1)^i f(m_0+(d-i)x)
	\end{align*}
	is equal to 
	\begin{align*}
		g'(x) &= \sum_{i=0}^d \binom{d}{i}(-1)^i (k-1-i)f'(m_0+(k-1-i)x)\\
		&\quad- \sum_{i=0}^d \binom{d}{i}(-1)^i (d-i)f'(m_0+(d-i)x)\\
		&= (k-1)\Delta_x^d f'(m_0+(k-d-1)x) + \sum_{i=0}^d \binom{d}{i}(-1)^{i+1} if'(m_0+(k-1-i)x)\\
		&\quad- d\Delta_x^d f'(m_0) - \sum_{i=0}^d \binom{d}{i}(-1)^{i+1} if'(m_0+(d-i)x).
	\end{align*}
	Using the equality $\binom{d}{i}i=d\binom{d-1}{i-1}$, 
	we have 
	\begin{align*}
		g'(x) &= (k-1)\Delta_x^d f'(m_0+(k-d-1)x) + d\sum_{i=1}^d \binom{d-1}{i-1}(-1)^{i+1} f'(m_0+(k-1-i)x)\\
		&\quad- d\Delta_x^d f'(m_0) - d\sum_{i=1}^d \binom{d-1}{i-1}(-1)^{i+1} f'(m_0+(d-i)x)\\
		&= (k-1)\Delta_x^d f'(m_0+(k-d-1)x) + d\sum_{i=0}^{d-1} \binom{d-1}{i}(-1)^i f'(m_0+(k-2-i)x)\\
		&\quad- d\Delta_x^d f'(m_0) - d\sum_{i=0}^{d-1} \binom{d-1}{i}(-1)^i f'(m_0+(d-1-i)x)\\
		&= (k-1)\Delta_x^d f'(m_0+(k-d-1)x) + d\Delta_x^{d-1} f'(m_0+(k-d-1)x)\\
		&\quad- d\Delta_x^d f'(m_0) - d\Delta_x^{d-1} f'(m_0)\\
		&= (k-1)\Delta_x^d f'(m_0+(k-d-1)x) + d\Delta_x^{d-1} f'(m_0+(k-d-1)x)\\
		&\quad- d\Delta_x^{d-1} f'(m_0+x).
	\end{align*}
	The mean value theorem implies that for all $x>0$ 
	there exist $\theta_1,\ldots,\theta_d\in(0,x)$, $\theta'_1\in[0,(k-d-2)x]$ and $\theta'_2\ldots,\theta'_d\in(0,x)$ 
	such that 
	\begin{align*}
		g'(x)
		&= (k-1)x\Delta_x^{d-1} f''(m_0+(k-d-1)x+\theta_1)\\
		&\quad+ d(k-d-2)x\Delta_x^{d-1} f''(m_0+x+\theta'_1)\\
		&= \cdots\\
		&= (k-1)x^d f^{(d+1)}(m_0+(k-d-1)x+\theta)\\
		&\quad+ d(k-d-2)x^d f^{(d+1)}(m_0+x+\theta')\\
		&\overset{(a)}{>} 0,
	\end{align*}
	where $\theta=\theta_1+\cdots+\theta_d$ and $\theta'=\theta'_1+\cdots+\theta'_d$; 
	the inequality $(a)$ follows from the fact that $f^{(d+1)}(y)>0$ for all $y\ge N_0$.
	Thus, $g'(x)$ is positive, 
	and $g(x)$ increases.
	Recalling that $(\lfloor{f(m_0+r_0j)}\rfloor)_{j=0}^{k-1}$ belongs to $\cP_{k,d}$, 
	we have 
	\begin{equation}
	\begin{split}
		2^d &= \Delta_{r_0}^d \lfloor{f(m_0+(k-d-1)r_0)}\rfloor - \Delta_{r_0}^d \lfloor{f(m_0)}\rfloor + 2^d\\
		&> \Delta_{r_0}^d f(m_0+(k-d-1)r_0) - \Delta_{r_0}^d f(m_0)\\
		&\overset{(b)}{\ge} \Delta_{R_0}^d f(m_0+(k-d-1)R_0) - \Delta_{R_0}^d f(m_0)\\
		&\overset{(c)}{=} (k-d-1)R_0\Delta_{R_0}^d f'(m_0+\theta_0)\\
		&\overset{(c)}{=} \cdots \overset{(c)}{=} (k-d-1)R_0^{d+1}f^{(d+1)}(m_0+\theta)\\
		&= 2^d c(1/\beta)\frac{f^{(d+1)}(m_0+\theta)}{f^{(d+1)}(m_0)},
	\end{split}\label{eq10}
	\end{equation}
	where the monotonicity of $g$ and the inequality $r_0\ge R_0$ have been used to obtain $(b)$; 
	the mean value theorem has been used to obtain $(c)$; 
	$\theta_0\in(0,(k-d-1)R_0)$, $\theta_1,\ldots,\theta_d\in(0,R_0)$, 
	and $\theta=\theta_0+\theta_1+\cdots+\theta_d$.
	Put $\beta_0=1+(k-1)R_0/m_0$.
	Since the inequality $\beta_0<\beta$ holds due to $m_0\ge N_0$, 
	it follows that 
	\begin{equation}
	\begin{split}
		f^{(d+1)}(m_0+\theta) &\overset{\text{(B1)}}{\ge} f^{(d+1)}(m_0+(k-1)R_0) = f^{(d+1)}(\beta_0m_0)\\
		&\ge f^{(d+1)}(\beta m_0) \overset{\text{(B3)}}{\ge} c(1/\beta)^{-1}f^{(d+1)}(m_0).
	\end{split}\label{eq11}
	\end{equation}
	Thus, \eqref{eq10} and \eqref{eq11} yield that 
	\begin{equation*}
		2^d > 2^d c(1/\beta)\frac{f^{(d+1)}(m_0+\theta)}{f^{(d+1)}(m_0)}
		\ge 2^d c(1/\beta)\frac{c(1/\beta)^{-1}f^{(d+1)}(m_0)}{f^{(d+1)}(m_0)} = 2^d,
	\end{equation*}
	which is a contradiction.
	Therefore, if $(\lfloor{f(n+rj)}\rfloor)_{j=0}^{k-1}\in\cP_{k,d}$, $n\ge N_0$ and $r\ge1$, 
	then $r < R(n)$.
	\par
	Next, we show Proposition~\ref{main1'} (limsup).
	Let $N\in\mathbb{N}$ be sufficiently large.
	Since the inequality 
	\begin{align*}
		&\quad \#\{ P\subset[1,N]_{\mathbb{Z}} : P\in\cP_{k,1},\ (\lfloor{f(n)}\rfloor)_{n\in P}\in\cP_{k,d} \}\\
		&\le \#\{ (n,r)\in[1,N]_{\mathbb{Z}}\times[1,N]_{\mathbb{Z}} : (\lfloor{f(n+rj)}\rfloor)_{j=0}^{k-1}\in\cP_{k,d} \}\\
		&\le \#\{ (n,r)\in[1,N_0]_{\mathbb{Z}}\times[1,N]_{\mathbb{Z}} \} + \#\{ (n,r)\in[N_0,N]_{\mathbb{Z}}\times[1,N]_{\mathbb{Z}} : r<R(n) \}\\
		&\le N_0N + \sum_{n=N_0}^N R(n)
	\end{align*}
	holds, it follows that 
	\begin{align*}
		&\quad \limsup_{N\to\infty}
		\frac{\#\{ P\subset[1,N]_{\mathbb{Z}} : P\in\cP_{k,1},\ (\lfloor{f(n)}\rfloor)_{n\in P}\in\cP_{k,d} \}}{Nf^{(d+1)}(N)^{-1/(d+1)}}\\
		&\le \limsup_{N\to\infty} \frac{f^{(d+1)}(N)^{1/(d+1)}}{N}\sum_{n=N_0}^N R(n)
		\overset{\text{(B1)}}{\le} \limsup_{N\to\infty} f^{(d+1)}(N)^{1/(d+1)}R(N)\\
		&= \Bigl( \frac{2^d c(1/\beta)}{k-d-1} \Bigr)^{1/(d+1)} < \infty.
	\end{align*}
\end{proof}

\begin{remark}\label{rem:ub}
	Let us consider the special case $f(x)=x^\alpha$ with $\alpha\in(d,d+1)$.
	Then we can take $c(\delta)$ in (B3) as $c(\delta)=\delta^{\alpha-d-1}$.
	Thus, 
	\begin{align*}
		&\quad \limsup_{N\to\infty}
		\frac{\#\{ P\subset[1,N]_{\mathbb{Z}} : P\in\cP_{k,1},\ (\lfloor{n^\alpha}\rfloor)_{n\in P}\in\cP_{k,d} \}}{((\alpha)_{d+1})^{-1/(d+1)}N^{2-\alpha/(d+1)}}\\
		&\le \limsup_{N\to\infty} \frac{f^{(d+1)}(N)^{1/(d+1)}}{N}\sum_{n=1}^N \Bigl( \frac{2^d c(1/\beta)}{k-d-1} \Bigr)^{1/(d+1)}f^{(d+1)}(n)^{-1/(d+1)}\\
		&= \limsup_{N\to\infty} N^{\alpha/(d+1)-2}\sum_{n=1}^N \Bigl( \frac{2^d \beta^{d+1-\alpha}}{k-d-1} \Bigr)^{1/(d+1)}n^{1-\alpha/(d+1)}\\
		&= \Bigl( \frac{2^d \beta^{d+1-\alpha}}{k-d-1} \Bigr)^{1/(d+1)}\frac{1}{2-\alpha/(d+1)}.
	\end{align*}
	The arbitrariness of $\beta>1$ yields 
	\begin{align*}
		&\quad \limsup_{N\to\infty}
		\frac{\#\{ P\subset[1,N]_{\mathbb{Z}} : P\in\cP_{k,1},\ (\lfloor{n^\alpha}\rfloor)_{n\in P}\in\cP_{k,d} \}}{N^{2-\alpha/(d+1)}}\\
		&\le \Bigl( \frac{2^d}{(\alpha)_{d+1}(k-d-1)} \Bigr)^{1/(d+1)}\frac{1}{2-\alpha/(d+1)}
		\eqqcolon \tilde{B}_{\alpha,k}.
	\end{align*}
	Therefore, the constant $B_{\alpha,k}$ in Theorem~\ref{main01} 
	is an arbitrary value in the interval $(\tilde{B}_{\alpha,k},\infty)$.
\end{remark}

\section{Further analysis: discrepancy and short intervals}\label{discrepancy}

In this section, we show Theorems~\ref{main2} and \ref{main3}.
These theorems are derived from the following proposition.

\begin{proposition}\label{mainprop}
	Let $\alpha\in(1,2)$ and $c>0$, 
	and let $k\ge3$ and $r\ge1$ be integers.
	Then, there exists $N_0=N_0(\alpha,k,r)\in\mathbb{N}$ such that 
	for all $N\in[N_0,\infty)_{\mathbb{Z}}$ and $L\in[1,cN]_{\mathbb{Z}}$, 
	\begin{equation}
	\begin{split}
		&\quad \frac{1}{L}\#\{ n\in[N,N+L)_{\mathbb{Z}} : (\lfloor{(n+rj)^\alpha}\rfloor)_{j=0}^{k-1}\in\cP_{k,1} \}
		- \frac{1}{k-1}\\
		&\ll_{\alpha,k,r,c}
		\begin{cases}
			N^{(\alpha-2)/6}(\log N)^{1/2} + N^{(2-\alpha)/2}/L^{1/2} & \alpha\in(1,2),\\
			N^{(\alpha-3)/14}(\log N)^{1/2} + N^{(2-\alpha)/2}/L^{1/2} & \alpha\in(1,3/2),\\
			(N^{(\alpha-3)/14} + N^{(3-\alpha)/6}/L^{1/2})(\log N)^{1/2} & \alpha\in[3/2,11/6).
		\end{cases}\label{shortintformula}
	\end{split}
	\end{equation}
\end{proposition}

\begin{remark}
	The first one of \eqref{shortintformula} is the best of the three cases 
	when $\alpha\in(1,5/4]\cup[11/6,2)$; 
	the second one of \eqref{shortintformula} is the best of the three cases when $\alpha\in(5/4,3/2)$.
	However, it depends on the growth rate of $L$ 
	whether the third one of \eqref{shortintformula} is the best of the three cases when $\alpha\in[3/2,11/6)$.
	For instance, if $L=N$ and $\alpha\in[3/2,11/6)$, 
	then the third one of \eqref{shortintformula} is the best of the three cases; 
	but if $\epsilon\in(0,(2-\alpha)/3)$, $L=N^{2-\alpha+\epsilon}$ and $\alpha\in[3/2,11/6)$, 
	then the first one of \eqref{shortintformula} is the best of the three cases.
\end{remark}

Proposition~\ref{mainprop} is an asymptotic formula for the number of integers $n\ge1$ in a short interval 
such that $(\lfloor{(n+rj)^\alpha}\rfloor)_{j=0}^{k-1}$ is an AP.
We prove Proposition~\ref{mainprop} at the end of this section.
Note that \eqref{shortintformula} is meaningless when $L=L(N)$ is sufficiently smaller than $N$.
This is because in the case, the right-hand side in \eqref{shortintformula} diverges to positive infinity as $N\to\infty$.
Before proving Proposition~\ref{mainprop}, 
let us show Theorems~\ref{main2} and \ref{main3} by using Proposition~\ref{mainprop}.

\begin{proof}[Proof of Theorem~$\ref{main2}$ assuming Proposition~$\ref{mainprop}$]
	Let $\alpha\in(1,2)$, and let $k\ge3$ and $r\ge1$ be integers.
	Also, define the set $\cQ$ as 
	\begin{equation}
		\cQ = \{ n\in\mathbb{N} : (\lfloor{(n+rj)^\alpha}\rfloor)_{j=0}^{k-1}\in\cP_{k,1} \}.
		\label{eq16}
	\end{equation}
	Then Proposition~\ref{mainprop} implies that 
	\begin{equation*}
		\frac{\#(\cQ\cap[x,2x))}{x} = \frac{1}{k-1} + O_{\alpha,k,r}(F_0(x)),
	\end{equation*}
	where 
	\[
	F_0(x) \coloneqq
	\begin{cases}
		x^{(\alpha-2)/6}(\log x)^{1/2} + x^{(1-\alpha)/2} & \alpha\in(1,5/4)\cup[11/6,2),\\
		x^{(\alpha-3)/14}(\log x)^{1/2} + x^{(1-\alpha)/2} & \alpha\in[5/4,3/2),\\
		(x^{(\alpha-3)/14} + x^{-\alpha/6})(\log x)^{1/2} & \alpha\in[3/2,11/6).
	\end{cases}
	\]
	Noting the ranges of $\alpha$, we have $F_0(x)\ll F(x)$, 
	where $F$ is defined in Theorem~\ref{main2}.
	Let $N\in\mathbb{N}$ be sufficiently large and 
	take $M\in\mathbb{N}$ with $2^M\le N<2^{M+1}$.
	Then 
	\begin{gather*}
		0 \le \frac{\#(\cQ\cap[1,N])}{N} - \frac{1}{N}\sum_{m=1}^M \#(\cQ\cap[2^{-m}N, 2^{1-m}N)) \le 2/N,\\
		\begin{split}
			&\quad \frac{1}{N}\sum_{m=1}^M \#(\cQ\cap[2^{-m}N, 2^{1-m}N))
			= \sum_{m=1}^M \frac{2^{-m}}{k-1} + O_{\alpha,k,r}\Bigl( \sum_{m=1}^M 2^{-m}F(2^{-m}N) \Bigr)\\
			&= \frac{1-2^{-M}}{k-1} + O_{\alpha,k,r}(F(N))
			= \frac{1}{k-1} + O_{\alpha,k,r}(1/N + F(N)).
		\end{split}
	\end{gather*}
	Therefore, Theorem~\ref{main2} holds.
\end{proof}

\begin{proof}[Proof of Theorem~$\ref{main3}$ assuming Proposition~$\ref{mainprop}$]
	Let $\alpha\in(1,2)$, and let $k\ge3$ and $r\ge1$ be integers.
	Define the set $\cQ$ as \eqref{eq16}.
	Thanks to the first inequality of \eqref{shortintformula}, 
	there exist constants $C=C(\alpha,k,r)>0$ and $N_0=N_0(\alpha,k,r)\in\mathbb{N}$ such that 
	for all $N\in[N_0,\infty)_{\mathbb{Z}}$ and $L\in[1,N]_{\mathbb{Z}}$, 
	\begin{equation}
		\abs{\frac{\#(\cQ\cap[N,N+L))}{L} - \frac{1}{k-1}} \le CE_0(N,L),
		\label{eq18}
	\end{equation}
	where $E_0(N,L) \coloneqq N^{(\alpha-2)/6}(\log N)^{1/2} + N^{(2-\alpha)/2}/L^{1/2}$.
	Without loss of generality, we may assume that 
	$N^{(\alpha-2)/6}(\log N)^{1/2} < 1/2C(k-1)$ for every integer $N\ge N_0$.
	Putting $L=L(N)=\lceil{4C^2(k-1)^2N^{2-\alpha}}\rceil$, 
	we have 
	\[
	E_0(N,L) < \frac{1}{2C(k-1)} + \frac{1}{2C(k-1)}
	= \frac{1}{C(k-1)}
	\]
	for every integer $N\ge N_0$.
	Therefore, for every integer $N\ge N_0$, 
	the left-hand side in \eqref{eq18} is less than $1/(k-1)$, 
	whence $\#(\cQ\cap[N,N+L))>0$.
	Finally, the length $L'=L'(N) \coloneqq \max\{N_0+L(N_0),L\}=O_{\alpha,k,r}(N^{2-\alpha})$ satisfies that 
	$\#(\cQ\cap[N,N+L'))>0$ for all $N\in\mathbb{N}$.
\end{proof}

To prove Proposition~\ref{mainprop}, 
we need to estimate the convergence speed of \eqref{eq13} for a uniformly distributed sequence.
For this purpose, let us define two kinds of discrepancies.
For a sequence $(\mathbf{x}_n)_{n=1}^N$ of $\mathbb{R}^d$, 
define the \textit{discrepancy} $D_N$ and \textit{isotropic discrepancy} $J_N$ as 
\begin{align*}
	D_N &= D(\mathbf{x}_1,\ldots,\mathbf{x}_N)\\
	&= \sup_{\substack{0\le a_i<b_i\le 1 \\ i\in[1,d]_{\mathbb{Z}}}}
	\abs{\frac{\#\bigl\{ n\in[1,N]_{\mathbb{Z}} : \{\mathbf{x}_n\}\in\prod_{i=1}^d [a_i,b_i) \bigr\}}{N} - \prod_{i=1}^d (b_i-a_i)},\\
	J_N &= J(\mathbf{x}_1,\ldots,\mathbf{x}_N)
	= \sup_{\substack{\cC\subset[0,1)^d\\ \text{convex}}} \abs{\frac{\#\bigl\{ n\in[1,N]_{\mathbb{Z}} : \{\mathbf{x}_n\}\in\cC \bigr\}}{N}- \mu(\cC)},
\end{align*}
where $\mu$ denotes the Lebesgue measure on $\mathbb{R}^d$.
Although the inequality $D_N\le J_N$ is trivial, 
the following reverse inequality holds \cite[Theorem~1.6, Chapter~2]{KN}: 
\begin{equation}\label{isotropic}
	J_N \leq (4d\sqrt{d}+1)D_N^{1/d}
\end{equation}
for every $d,N\in\mathbb{N}$ and $\mathbf{x}_1,\ldots,\mathbf{x}_N\in\mathbb{R}^d$.
Thanks to \eqref{isotropic}, 
it suffices to give an upper bound for the discrepancy 
in order to estimate the convergence speed of \eqref{eq13}.

Now, the following inequality is useful to evaluate discrepancies.

\begin{lemma}[Koksma-Sz\"{u}sz \cite{Koksma,Szusz}]\label{ETK}
	For all $d,L,H\in\mathbb{N}$ and $\mathbf{x}_1,\ldots, \mathbf{x}_L \in \mathbb{R}^d$, 
	\[
	D(\mathbf{x}_1,\ldots,\mathbf{x}_L)
	\ll_d \frac{1}{H} + \sum_{\substack{0<\|\mathbf{h}\|_\infty \le H \\ \mathbf{h}\in\mathbb{Z}^d}}
	\frac{1}{u(\mathbf{h})}\abs{\frac{1}{L}\sum_{n=1}^L e(\langle{\mathbf{h},\mathbf{x}_n}\rangle)}, 
	\]
	where $u(\mathbf{h}) \coloneqq \prod_{i=1}^d \max\{1,|h_i|\}$.
\end{lemma}

The above inequality is sometimes referred as the Erd\H{o}s-Tur\'an-Koksma inequality.
Thanks to Lemma~\ref{ETK}, 
it suffices to evaluate exponential sums in order to find upper bounds for discrepancies.
Next, let us state the following lemmas that are used to evaluate exponential sums.

\begin{lemma}[Kusmin-Landau]\label{1stderiv}
	Let $\cI$ be an interval of $\mathbb{R}$, 
	and $f\colon \cI\to\mathbb{R}$ be a $C^1$ function such that $f'$ is monotone.
	If $\lambda_1>0$ satisfies that 
	\[
	\lambda_1 \leq \min\{|f'(x)-n| : n\in\mathbb{Z}\}
	\] 
	for all $x\in\cI$, then 
	\[
	\sum_{n\in I_{\mathbb{Z}}} e(f(n)) \ll \lambda_1^{-1}. 
	\]
\end{lemma}

\begin{lemma}[van der Corput]\label{2ndderiv}
	Let $\cI$ be an interval of $\mathbb{R}$ and 
	$f\colon \cI\to\mathbb{R}$ be a $C^2$ function, 
	and let $c\ge1$.
	If $\lambda_2>0$ satisfies that 
	\[
	\lambda_2 \le |f''(x)| \le c\lambda_2
	\] 
	for all $x\in\cI$, then 
	\[
	\sum_{n\in\cI_{\mathbb{Z}}} e(f(n))
	\ll_c \abs{\cI}\lambda_2^{1/2} + \lambda_2^{-1/2},
	\]
	where $|\cI|$ denotes the length of the interval $\cI$.
\end{lemma}

\begin{lemma}[Sargos-Gritsenko]\label{3rdderiv}
	Let $\cI$ be an interval of $\mathbb{R}$ and 
	$f\colon \cI\to\mathbb{R}$ be a $C^3$ function, 
	and let $0<c_1<c_2$.
	If $\lambda_3\in(0,1)$ satisfies that 
	\[
	c_1\lambda_3 \le |f'''(x)| \le c_2\lambda_3
	\]
	for all $x\in\cI$, then 
	\[
	\sum_{n\in\cI_{\mathbb{Z}}} e(f(n))
	\ll_{c_1,c_2} \abs{\cI}\lambda_3^{1/6} + \lambda_3^{-1/3}.
	\]
\end{lemma}

Lemmas~\ref{1stderiv} and \ref{2ndderiv} are called the first and second derivative tests, respectively.
One can see their proofs in \cite[Theorems~2.1 and 2.2]{GK}.
Lemma~\ref{3rdderiv} was shown by Sargos \cite{Sargos} and Grisenko \cite{Grisenko} independently.
Using Lemmas~\ref{1stderiv}--\ref{3rdderiv}, 
we evaluate discrepancies.

\begin{lemma}\label{mainlemma}
	Let $\alpha\in(1,2)$, $r\in\mathbb{N}$, and $c>0$.
	Then, there exists $N_0=N_0(\alpha,r)\in\mathbb{N}$ such that 
	for all $N\in[N_0,\infty)_{\mathbb{Z}}$ and $L\in[1,cN]_{\mathbb{Z}}$, 
	the discrepancy $D(N,L)$ of the sequence $\bigl( (n^\alpha, r\alpha n^{\alpha-1}) \bigr)_{n=N}^{N+L-1}$ satisfies 
	\begin{equation*}
		D(N,L) \ll_{\alpha,c}
		\begin{cases}
			N^{(\alpha-2)/3}\log N + N^{2-\alpha}/L & \alpha\in(1,2),\\
			N^{(\alpha-3)/7}\log N + N^{2-\alpha}/L & \alpha\in(1,3/2),\\
			(N^{(\alpha-3)/7} + N^{(3-\alpha)/3}/L)\log N & \alpha\in[3/2,11/6).
		\end{cases}
	\end{equation*}
\end{lemma}
\begin{proof}
	Let $f(x)=x^\alpha$.
	Lemma~\ref{ETK} with $d=2$ implies that for all $L,N,H\in\mathbb{N}$, 
	\[
	D(N,L) \ll \frac{1}{H} + \sum_{\substack{|h_0|,|h_1|\le H \\ (h_0,h_1)\neq(0,0)}} \frac{1}{u(h_0,h_1)}
	\abs{\frac{1}{L}\sum_{n=N}^{N+L-1} e(h_0f(n)+h_1rf'(n))}.
	\]
	Taking an integer 
	\begin{align*}
		&\quad N_0 = N_0(\alpha,r)\\
		&\ge \max\{ (2r)^{3/(1+\alpha)}, 2^{3/(2-\alpha)}, (4r)^{3/2(2-\alpha)},\ 
		(2r)^{7/(4+\alpha)}, 2^{7/(3-\alpha)}, (4r)^{7/(11-6\alpha)} \},
	\end{align*}
	we evaluate the right-hand side above in two ways.
	\par\setcounter{count}{0}
	\noindent\textbf{Step~\num.}
	Let us show that for all $N\in[N_0,\infty)_{\mathbb{Z}}$ and $L\in[1,cN]_{\mathbb{Z}}$, 
	\begin{equation}
		D(N,L) \ll_{\alpha,c} N^{(\alpha-2)/3}\log N + N^{2-\alpha}/L. \label{eq15}
	\end{equation}
	Take $N\in[N_0,\infty)_{\mathbb{Z}}$ and $L\in[1,cN]_{\mathbb{Z}}$ arbitrarily, 
	and put $H=\lfloor{N^{(2-\alpha)/3}}\rfloor$.
	Then, note that $rH/N\le rN^{-(1+\alpha)/3}\le rN_0^{-(1+\alpha)/3}\le 1/2$ 
	and $\log H\ge\log2$.
	Consider the case when $|h_0|,|h_1|\le H$ and $h_0\not=0$.
	When $x\in[N,N+L-1]$, the function $g(x) = h_0f(x) + h_1rf'(x)$ satisfies that 
	\begin{align*}
		\abs{g''(x)} &\le \abs{h_0}f''(x)(1 + rH\abs{f'''(x)/f''(x)})\\
		&\ll \abs{h_0}N^{\alpha-2}(1+rH/N)
		\ll \abs{h_0}N^{\alpha-2},\\
		\abs{g''(x)} &\ge \abs{h_0}f''(x)(1 - rH\abs{f'''(x)/f''(x)})\\
		&\gg_\alpha \abs{h_0}(N+L)^{\alpha-2}(1-rH/N)
		\gg_c \abs{h_0}N^{\alpha-2}.
	\end{align*}
	Thus, Lemma~\ref{2ndderiv} implies that 
	\[
	\frac{1}{L}\sum_{n=N}^{N+L-1} e(h_0f(n)+h_1rf'(n))
	\ll_{\alpha,c} \abs{h_0}^{1/2}N^{(\alpha-2)/2} + \abs{h_0}^{-1/2}N^{(2-\alpha)/2}/L.
	\]
	Therefore, it follows that 
	\begin{align*}
		&\quad \sum_{\substack{|h_0|,|h_1|\le H \\ h_0\neq0}} \frac{1}{u(h_0,h_1)}
		\abs{\frac{1}{L}\sum_{n=N}^{N+L-1} e(h_0f(n)+h_1rf'(n))}\\
		&\ll_{\alpha,c} \sum_{\substack{|h_0|,|h_1|\le H \\ h_0\neq0}}
		\frac{\abs{h_0}^{1/2}N^{(\alpha-2)/2} + \abs{h_0}^{-1/2}N^{(2-\alpha)/2}/L}{u(h_0,h_1)}\\
		&\ll \Bigl( \sum_{h_1=1}^H \frac{1}{h_1} + 1 \Bigr)
		\sum_{h_0=1}^H \bigl( h_0^{-1/2}N^{(\alpha-2)/2} + h_0^{-3/2}N^{(2-\alpha)/2}/L \bigr)\\
		&\ll (\log H)(H^{1/2}N^{(\alpha-2)/2} + N^{(2-\alpha)/2}/L)
		\ll (N^{(\alpha-2)/3} + N^{(2-\alpha)/2}/L)\log N.
	\end{align*}
	\par
	Next, consider the case when $1\le|h_1|\le H$ and $h_0=0$.
	When $x\in[N,N+L-1]$, the function $g(x) = h_1rf'(x)$ satisfies that 
	\begin{gather*}
		\abs{g'(x)} = r\abs{h_1}f''(x) \le 2rHN^{\alpha-2} \le 2rN^{(2/3)(\alpha-2)} \le 2rN_0^{(2/3)(\alpha-2)} \le 1/2,\\
		\abs{g'(x)} \gg_\alpha \abs{h_1}(N+L)^{\alpha-2}
		\gg_c \abs{h_1}N^{\alpha-2}.
	\end{gather*}
	This yields that $\min\{|g'(x)-m| : m\in\mathbb{Z}\} = |g'(x)|$ for all $x\in[N,N+L-1]$.
	Thus, Lemma~\ref{1stderiv} implies that 
	\[
	\frac{1}{L}\sum_{n=N}^{N+L-1} e(h_1rf'(n))
	\ll_{\alpha,c} \abs{h_1}^{-1}N^{2-\alpha}/L.
	\]
	Therefore, it follows that 
	\begin{equation}
	\begin{split}
		&\quad \sum_{\substack{1\le|h_1|\le H \\ h_0=0}} \frac{1}{u(h_0,h_1)}
		\abs{\frac{1}{L}\sum_{n=N}^{N+L-1} e(h_0f(n)+h_1rf'(n))}\\
		&\ll_{\alpha,c} \sum_{1\le|h_1|\le H} \frac{\abs{h_1}^{-1}N^{2-\alpha}/L}{\abs{h_1}}
		\ll N^{2-\alpha}/L.
	\end{split}\label{eq14}
	\end{equation}
	Summarizing the above two cases, we have 
	\begin{align*}
		D(N,L) &\ll \frac{1}{H} + \sum_{\substack{|h_0|,|h_1|\le H \\ (h_0,h_1)\neq(0,0)}} \frac{1}{u(h_0,h_1)}
		\abs{\frac{1}{L}\sum_{n=N}^{N+L-1} e(h_0f(n)+h_1rf'(n))}\\
		&\ll_{\alpha,c} N^{(\alpha-2)/3} + (N^{(\alpha-2)/3} + N^{(2-\alpha)/2}/L)\log N + N^{2-\alpha}/L\\
		&\ll N^{(\alpha-2)/3}\log N + N^{2-\alpha}/L,
	\end{align*}
	which is just \eqref{eq15}.
	\par
	\noindent\textbf{Step~\num.}
	Assume $\alpha\in(1,11/6)$.
	Let us show that for all $N\in[N_0,\infty)_{\mathbb{Z}}$ and $L\in[1,cN]_{\mathbb{Z}}$, 
	\begin{equation}
		D(N,L) \ll_{\alpha,c}
		\begin{cases}
			N^{(\alpha-3)/7}\log N + N^{2-\alpha}/L & \alpha\in(1,3/2),\\
			(N^{(\alpha-3)/7} + N^{(3-\alpha)/3}/L)\log N & \alpha\in[3/2,11/6).
		\end{cases}\label{eq15'}
	\end{equation}
	Take $N\in[N_0,\infty)_{\mathbb{Z}}$ and $L\in[1,cN]_{\mathbb{Z}}$ arbitrarily, 
	and put $H=\lfloor{N^{(3-\alpha)/7}}\rfloor$.
	Then, note that $rH/N\le rN^{-(4+\alpha)/7}\le rN_0^{-(4+\alpha)/7}\le 1/2$ 
	and $\log H\ge\log2$.
	Consider the case when $|h_0|,|h_1|\le H$ and $h_0\not=0$.
	When $x\in[N,N+L-1]$, the function $g(x) = h_0f(x) + h_1rf'(x)$ satisfies that 
	\begin{align*}
		\abs{g'''(x)} &\le \abs{h_0f'''(x)}(1 + rH\abs{f''''(x)/f'''(x)})\\
		&\ll \abs{h_0}N^{\alpha-3}(1+rH/N)
		\ll \abs{h_0}N^{\alpha-3},\\
		\abs{g'''(x)} &\ge \abs{h_0}f'''(x)(1 - rH\abs{f''''(x)/f'''(x)})\\
		&\gg_\alpha \abs{h_0}(N+L)^{\alpha-3}(1-rH/N)
		\gg_c \abs{h_0}N^{\alpha-3}.
	\end{align*}
	Since $0 < \abs{h_0}N^{\alpha-3} \le HN^{\alpha-3} \le N^{(6/7)(\alpha-3)} < 1$, 
	Lemma~\ref{3rdderiv} implies that 
	\[
	\frac{1}{L}\sum_{n=N}^{N+L-1} e(h_0f(n)+h_1rf'(n))
	\ll_{\alpha,c} \abs{h_0}^{1/6}N^{(\alpha-3)/6} + \abs{h_0}^{-1/3}N^{(3-\alpha)/3}/L.
	\]
	Therefore, it follows that 
	\begin{align*}
		&\quad \sum_{\substack{|h_0|,|h_1|\le H \\ h_0\neq0}} \frac{1}{u(h_0,h_1)}
		\abs{\frac{1}{L}\sum_{n=N}^{N+L-1} e(h_0f(n)+h_1rf'(n))}\\
		&\ll_{\alpha,c} \sum_{\substack{|h_0|,|h_1|\le H \\ h_0\neq0}}
		\frac{\abs{h_0}^{1/6}N^{(\alpha-3)/6} + \abs{h_0}^{-1/3}N^{(3-\alpha)/3}/L}{u(h_0,h_1)}\\
		&\ll \Bigl( \sum_{h_1=1}^H \frac{1}{h_1} + 1 \Bigr)
		\sum_{h_0=1}^H \bigl( h_0^{-5/6}N^{(\alpha-3)/6} + h_0^{-4/3}N^{(3-\alpha)/3}/L \bigr)\\
		&\ll (\log H)(H^{1/6}N^{(\alpha-3)/6} + N^{(3-\alpha)/3}/L)
		\ll (N^{(\alpha-3)/7} + N^{(3-\alpha)/3}/L)\log N.
	\end{align*}
	\par
	Next, consider the case when $1\le|h_1|\le H$ and $h_0=0$.
	When $x\in[N,N+L-1]$, the function $g(x) = h_1rf'(x)$ satisfies that 
	\begin{gather*}
		\abs{g'(x)} = r\abs{h_1}f''(x) \le 2rHN^{\alpha-2} \le 2rN^{(6\alpha-11)/7} \le 2rN_0^{(6\alpha-11)/7} \le 1/2,\\
		\abs{g'(x)} \gg_{\alpha,c} \abs{h_1}N^{\alpha-2}.
	\end{gather*}
	This yields that $\min\{|g'(x)-m| : m\in\mathbb{Z}\} = |g'(x)|$ for all $x\in[N,N+L-1]$.
	From the same calculation as Step~1, the inequality \eqref{eq14} follows.
	Summarizing the above two cases, we have 
	\begin{align*}
		D(N,L) &\ll \frac{1}{H} + \sum_{\substack{|h_0|,|h_1|\le H \\ (h_0,h_1)\neq(0,0)}} \frac{1}{u(h_0,h_1)}
		\abs{\frac{1}{L}\sum_{n=N}^{N+L-1} e(h_0f(n)+h_1rf'(n))}\\
		&\ll_{\alpha,c} N^{(\alpha-3)/7} + (N^{(\alpha-3)/7} + N^{(3-\alpha)/3}/L)\log N + N^{2-\alpha}/L\\
		&\ll (N^{(\alpha-3)/7} + N^{(3-\alpha)/3}/L)\log N + N^{2-\alpha}/L\\
		&\ll
		\begin{cases}
			N^{(\alpha-3)/7}\log N + N^{2-\alpha}/L & \alpha\in(1,3/2),\\
			(N^{(\alpha-3)/7} + N^{(3-\alpha)/3}/L)\log N & \alpha\in[3/2,11/6),
		\end{cases}
	\end{align*}
	which is just \eqref{eq15'}.
	\par
	Finally, combining \eqref{eq15} and \eqref{eq15'}, 
	we obtain Lemma~\ref{mainlemma}.
\end{proof}

\begin{proof}[Proof of Proposition~$\ref{mainprop}$]
	Take $N_0=N_0(\alpha,r)\in\mathbb{N}$ in Lemma~\ref{mainlemma}.
	Let $f(x)=x^\alpha$, 
	\[
	N'_0 = N'_0(\alpha,k,r)
	= \max\bigl\{ N_0, \bigl\lceil{\bigl( r^2(k-1)^2\alpha(\alpha-1) \bigr)^{1/(2-\alpha)}}\bigr\rceil \bigr\},
	\]
	$N\in[N'_0,\infty)_{\mathbb{Z}}$ and $L\in[1,cN]_{\mathbb{Z}}$.
	Then 
	\begin{equation*}
		\epsilon = \epsilon(N) \coloneqq \frac{r^2(k-1)^2}{2}f''(N) \in (0,1/2).
	\end{equation*}
	The discrepancy and isotropic discrepancy 
	of the sequence $\bigl( (a_0(n),a_1(n)) \bigr)_{n=N}^{N+L-1}$ 
	are denoted by $D(N,L)$ and $J(N,L)$ respectively, 
	where $a_0(n)$ and $a_1(n)$ are defined by \eqref{eq17} with $d=1$.
	Note that $a_0(n)=f(n)$ and $a_1(n)=rf'(n)$.
	Also, define the set $\cQ$ as \eqref{eq16}.
	Recall the proof of Proposition~\ref{main0'}.
	The sets $\cC_{k,2}^{\mp}(\epsilon)$ defined by \eqref{eqC-} and \eqref{eqC+} with $d=1$ 
	satisfy the inclusion relations 
	\begin{gather}
		\bigcup_{s_1\in\mathbb{Z}}
		\{ n\in[N,\infty)_{\mathbb{Z}} : (\{a_0(n)\},\{a_1(n)\}+s_1)\in\cC_{k,2}^{-}(\epsilon) \} \subset \cQ,\nonumber\\
		\cQ\cap[N,\infty) \subset \bigcup_{s_1\in\mathbb{Z}}
		\{ n\in\mathbb{N} : (\{a_0(n)\},\{a_1(n)\}+s_1)\in\cC_{k,2}^{+}(\epsilon) \}. \label{eq12}
	\end{gather}
	Thus, we have that 
	\begin{align*}
		&\quad \frac{\#(\cQ\cap[N,N+L))}{L}
		\ge \sum_{s_1\in\mathbb{Z}}
		\frac{\#\{ n\in[N,N+L)_{\mathbb{Z}} : (\{a_0(n)\},\{a_1(n)\}+s_1)\in\cC_{k,2}^{-}(\epsilon) \}}{L}\\
		&\ge \sum_{s_1\in\mathbb{Z}} \biggl( \mu\Bigl( \cC_{k,2}^{-}(\epsilon)\cap\bigl( [0,1)\times[s_1,s_1+1) \bigr) \Bigr) - J(N,L) \biggr)
		\ge \mu(\cC_{k,2}^{-}(\epsilon)) - C_k^{-}J(N,L)
	\end{align*}
	and 
	\begin{align*}
		&\quad \frac{\#(\cQ\cap[N,N+L))}{L}
		\le \sum_{s_1\in\mathbb{Z}}
		\frac{\#\{ n\in[N,N+L)_{\mathbb{Z}} : (\{a_0(n)\},\{a_1(n)\}+s_1)\in\cC_{k,2}^{+}(\epsilon) \}}{L}\\
		&\le \sum_{s_1\in\mathbb{Z}} \biggl( \mu\Bigl( \cC_{k,2}^{+}(\epsilon)\cap\bigl( [0,1)\times[s_1,s_1+1) \bigr) \Bigr) + J(N,L) \biggr)
		\le \mu(\cC_{k,2}^{+}(\epsilon)) + C_k^{+}J(N,L)
	\end{align*}
	for some $C_k^{\mp}\in\mathbb{N}$, 
	since all the above sums are finite sums.
	(Indeed, we can take $C_k^{\mp}=2$, but this fact is not used here).
	Now, the sets $\cC_{k,2}^{\mp}(\epsilon)$ are simplified as 
	\begin{align*}
		\cC_{k,2}^{-}(\epsilon) &= \{ (y_0,y_1)\in\mathbb{R}^2 : 0\le y_0<1,\ \epsilon\le y_0+(k-1)y_1<1-\epsilon \},\\
		\cC_{k,2}^{+}(\epsilon) &= \{ (y_0,y_1)\in\mathbb{R}^2 : 0\le y_0<1,\ -\epsilon\le y_0+(k-1)y_1<1+\epsilon \},
	\end{align*}
	whence $\mu(\cC_{k,2}^{\mp}(\epsilon))=(1\mp2\epsilon)/(k-1)$.
	Thus, 
	\[
	\abs{\frac{\#(\cQ\cap[N,N+L))}{L} - \frac{1}{k-1}}
	\le \frac{2\epsilon}{k-1} + \max\{C_k^{\mp}\}\cdot J(N,L).
	\]
	Using the inequality \eqref{isotropic} and Lemma~\ref{mainlemma}, 
	we obtain 
	\begin{align*}
		&\quad \abs{\frac{\#(\cQ\cap[N,N+L))}{L} - \frac{1}{k-1}}
		\le \frac{2\epsilon}{k-1} + \max\{C_k^{\mp}\}\cdot2(8\sqrt{2}+1)D(N,L)^{1/2}\\
		&\ll_{k,r} N^{\alpha-2} + D(N,L)^{1/2}\\
		&\ll_{\alpha,c} 
		\begin{cases}
			N^{(\alpha-2)/6}(\log N)^{1/2} + N^{(2-\alpha)/2}/L^{1/2} & \alpha\in(1,2),\\
			N^{(\alpha-3)/14}(\log N)^{1/2} + N^{(2-\alpha)/2}/L^{1/2} & \alpha\in(1,3/2),\\
			(N^{(\alpha-3)/14} + N^{(3-\alpha)/6}/L^{1/2})(\log N)^{1/2} & \alpha\in[3/2,11/6),
		\end{cases}
	\end{align*}
	where the inequality $(x+y)^{1/2} \le x^{1/2}+y^{1/2}$ for $x,y\ge0$ 
	has been used to obtain the last inequality.
\end{proof}

\section{Future work}\label{future}

We have investigated distributions of finite sequences represented by polynomials in $\mathrm{PS}(\alpha)$, 
and especially done the case $\alpha\in(1,2)$ in detail.
We have not proved the convergence in the proof of Theorem~\ref{main01}, 
but the middle-hand side in \eqref{eq02} divided by $N^{2-\alpha/(d+1)}$ probably converges to some positive number as $N\to\infty$.
It is a future work.
As other natural questions, we have the positive-density version and prime-number version.

\begin{question}[Positive-density version]\label{Q1}
	Let $d\in\mathbb{N}$ and $\alpha\in(d,d+1)$; 
	let $A\subset\mathbb{N}$ be a set with positive density, 
	and $k\ge d+2$ and $r\ge1$ be integers.
	Then does 
	\begin{equation}
		\#\{ P\subset A\cap[1,N] : P\in\cP_{k,1},\ (\lfloor{n^\alpha}\rfloor)_{n\in P}\in\cP_{k,d} \}
		\asymp N^{2-\alpha/(d+1)} \quad (N\to\infty) \label{eqQ1}
	\end{equation}
	hold?
\end{question}

\begin{question}[Prime-number version]\label{Q2}
	How about the case when $A$ in Question~\ref{Q1} is replaced with the set of all prime numbers?
	In this case, what is suitable as the right-hand side in \eqref{eqQ1}?
\end{question}

Actually, we can replace the first term $n$ in \eqref{eq:main0} with a prime number $p$: 
for every $f\in\cH$ that satisfies the same assumptions as Theorem~$\ref{main0}$, 
\begin{equation}
	\lim_{N\to\infty} \frac{1}{\pi(N)}
	\#\{ p\in[n_0,N]_{\mathbb{Z}} : (\lfloor{f(p+rj)}\rfloor)_{j=0}^{k-1}\in\cP_{k,d} \}
	= \mu(\cC_{k,d+1}), \label{eq:main0p}
\end{equation}
where $\pi(N)$ denotes the number of prime numbers less than or equal to $N$.
The proof of this statement is the same as that of Theorem~\ref{main0} 
because for every subpolynomial $f\in\cH$ defined on the interval $[n_0,\infty)$, 
the sequence $(f(p))_{p\,\text{prime}\ge n_0}$ is uniformly distributed modulo $1$ 
if and only if $(f(n))_{n=n_0}^\infty$ is uniformly distributed modulo $1$ \cite{BKS}.
In \eqref{eq:main0p}, it is only guaranteed that the first term $p$ is prime.
In order to make all terms $p,p+r,\ldots,p+(k-1)r$ prime, 
we need to study whether $(f(p))_{p\in\cS_{k,r}\cap[n_0,\infty)}$ is uniformly distributed modulo $1$ or not, 
where $\cS_{k,r}$ is the set of all prime numbers $p$ such that all $p,p+r,\ldots,p+(k-1)r$ are prime.
Of course, $r$ must be restricted to some extent depending on $k$.
The set $\cS_{k,r}$ is related to 
twin prime pairs (when $(k,r)=(2,2)$), sexy prime triplets (when $(k,r)=(3,6)$), and generally prime $k$-tuples.
It is known that there exists an even number $r$ such that $\cS_{2,r}$ is infinite \cite{polymath,Maynard}, 
but it is still open whether $\cS_{k,r}$ is infinite for general $k$ and admissible $r$.

Finally, we focus on an asymptotic formula 
when $\alpha$ runs over the interval $(1,2)$.

\begin{question}[Asymptotic formula when $\alpha$ running]\label{Q3}
	Fix a sufficiently large $N\in\mathbb{N}$ and integers $k\ge3$ and $r\ge1$.
	Let 
	\[
	D_{N,k,r}(\alpha) = \frac{1}{N}\#\{ n\in[1,N]_{\mathbb{Z}} : (\lfloor{(n+rj)^\alpha}\rfloor)_{j=0}^{k-1}\in\cP_{k,1} \}.
	\]
	Can we find any asymptotic formulas of $D_{N,k,r}(\alpha)$ 
	when $\alpha$ runs over the interval $(1,2)$?
\end{question}

\begin{figure}[t]
	\centering
    \includegraphics[width=15.0cm]{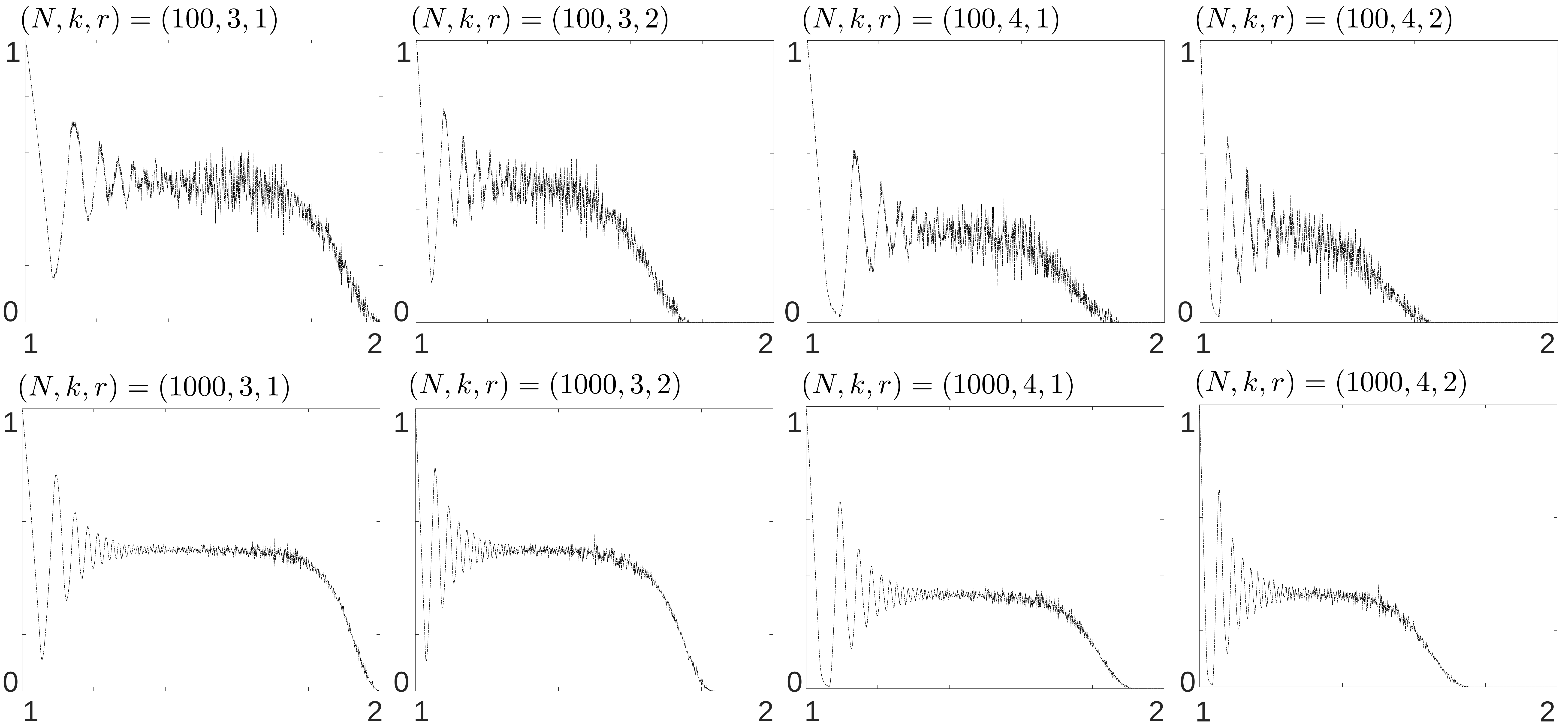}
	\caption{The behavior of $D_{N,k,r}(\alpha)$ for $(N,k,r)\in\{100,1000\}\times\{3,4\}\times\{1,2\}$.
	The abscissa and ordinate denote values of $\alpha$ and $D_{N,k,r}(\alpha)$, respectively.}\label{figure1}
\end{figure}

Figure~\ref{figure1} illustrates the behavior of $D_{N,k,r}(\alpha)$ by numerical computation, 
where the points $(\alpha, D_{N,k,r}(\alpha))$ are plotted 
for all $\alpha\in\{1+0.001i : i\in[0,1000]_{\mathbb{Z}}\}$.
In view of this figure, 
$D_{N,k,r}(\alpha)$ would be approximated by the sum of continuous waves and discrete errors.
In order to theoretically observe a phenomenon like this figure, 
it is probably needed to further analyze 
the distribution of the sequence $\bigl( (n^\alpha, \alpha n^{\alpha-1}) \bigr)_{n=1}^N$ modulo $1$.

\section*{Acknowledgment}
KS is financially supported by JSPS KAKENHI Grant Number JP19J20878.
YY is financially supported by JSPS KAKENHI Grant Number JP19J20161.

\appendix
\section{Optimality of the growth rate $O_{\alpha,k,r}(x^{2-\alpha})$}\label{optimality}

Throughout this appendix, let $f(x)=x^\alpha$.
As stated in Theorem~\ref{main2}, 
the relation $L_{\alpha,k,r}(x)=O_{\alpha,k,r}(x^{2-\alpha})$ holds.
We show that the growth rate $O_{\alpha,k,r}(x^{2-\alpha})$ is best for every $k\ge4$ in the following meaning.

\begin{proposition}\label{prop:k4}
	For all $\alpha\in(1,2)$ and all integers $k\ge4$ and $r\ge1$, 
	\begin{equation}
		\limsup_{x\to\infty} \frac{L_{\alpha,k,r}(x)}{x^{2-\alpha}}
		\ge \frac{k-3}{\alpha(\alpha-1)r(k-1)}.
		\label{bestL}
	\end{equation}
\end{proposition}
\begin{proof}
	Let $k\ge4$ and $r\ge1$ be integers, 
	and let $\alpha\in(1,2)$ and $\beta\in(0,k-3)$.
	Since $(rf'(n))_{n=1}^\infty$ is uniformly distributed modulo $1$ and 
	the inequality $1/(k-1) < 1 - (\beta+1)/(k-1)$ holds, 
	there exist infinitely many $N\in\mathbb{N}$ such that 
	\begin{equation}\label{eqapp1}
		\frac{1}{k-1} \le \{rf'(N)\} \le 1 - \frac{\beta+1}{k-1}.
	\end{equation}
	Take a sufficiently large $N\in\mathbb{N}$ that satisfies \eqref{eqapp1} and 
	\begin{equation}
		\epsilon = \epsilon(N) \coloneqq \frac{(k-1)^2r^2}{2}f''(N) \in (0,1). \label{eqapp2}
	\end{equation}
	Now, define the set $\cQ$ as \eqref{eq16}, 
	and take $m\in[0,L_{\alpha,k,r}(N)]_{\mathbb{Z}}$ such that $N+m\in\cQ$.
	Recall the proof of Proposition~\ref{main0'}.
	The set 
	\begin{equation}
		\cC_{k,2}^{+}(\epsilon) = \{ (y_0,y_1)\in\mathbb{R}^2 : 0\le y_0<1,\ -\epsilon\le y_0+(k-1)y_1<1 \}
		\label{eq20}
	\end{equation}
	satisfies the inclusion relation \eqref{eq12}, 
	where $a_0(n)$ and $a_1(n)$ are defined by \eqref{eq17} with $d=1$.
	Note that $a_0(n)=f(n)$ and $a_1(n)=rf'(n)$.
	Due to \eqref{eq12}, 
	the vector $(\{f(N+m)\}, \{rf'(N+m)\}+s_1)$ lies in $\cC_{k,2}^{+}(\epsilon)$ for some $s_1\in\mathbb{Z}$.
	The integer $s_1$ is equal to $0$ or $-1$, 
	which is proved at the end of this proof.
	\par
	If $s_1=-1$, then the inequalities $-\epsilon\le \{f(N+m)\}+(k-1)(\{rf'(N+m)\}-1)<1$ and \eqref{eqapp1} 
	and the mean value theorem imply that 
	\begin{align*}
		1-\frac{1+\epsilon}{k-1} &\le \{rf'(N+m)\}
		\le \{rf'(N)\} + rmf''(N)\\
		&\le 1-\frac{\beta+1}{k-1} + rL_{\alpha,k,r}(N)f''(N),
	\end{align*}
	whence $L_{\alpha,k,r}(N)f''(N) \ge (\beta-\epsilon)/r(k-1)$.
	If $s_1=0$, then the inequalities $-\epsilon\le \{f(N+m)\}+(k-1)\{rf'(N+m)\}<1$ and \eqref{eqapp1} yield 
	\[
	\{rf'(N+m)\} < \frac{1}{k-1} \le \{rf'(N)\} \le 1-\frac{\beta+1}{k-1}.
	\]
	Since $f'$ and $f''$ are increasing and decreasing functions respectively, 
	the mean value theorem implies that 
	\[
	\frac{\beta+1}{k-1} \le rf'(N+m) - rf'(N)
	\le rmf''(N) \le rL_{\alpha,k,r}(N)f''(N),
	\]
	whence $L_{\alpha,k,r}(N)f''(N) \ge \beta/r(k-1)$.
	Since $\epsilon=\epsilon(N)$ vanishes as $N\to\infty$, 
	it turns out that 
	\[
	\limsup_{x\to\infty} \frac{L_{\alpha,k,r}(x)}{x^{2-\alpha}} \ge \frac{\beta}{\alpha(\alpha-1)r(k-1)}.
	\]
	Letting $\beta\to k-3$, we obtain \eqref{bestL}.
	\par
	We show that if $(x_0,x_1+s_1)\in\cC_{k,2}^{+}(\epsilon)$, 
	$(x_0,x_1)\in[0,1)^2$ and $s_1\in\mathbb{Z}$, then $s_1\in\{0,-1\}$.
	(The assumption $k\ge3$ suffices here.)
	The definition of $\cC_{k,2}^{+}(\epsilon)$ yields that 
	\begin{gather*}
		(k-1)s_1 \le x_0+(k-1)(x_1+s_1) < 1+\epsilon < 2,\\
		0 \le x_0+(k-1)(x_1+s_1) < 1+(k-1)(1+s_1),
	\end{gather*}
	whence $-3/2\le-k/(k-1)<s_1<2/(k-1)\le1$.
	Therefore, the integer $s_1$ is equal to $0$ or $-1$.
\end{proof}

When $k=3$, the above proof does not work well, 
since there does not exist $N\in\mathbb{N}$ satisfying \eqref{eqapp1}.
The relation $L_{\alpha,3,r}(x)=O_{\alpha,r}(x^{1-\alpha/2})$ probably holds, 
but we do not have its proof.
However, if $L_{\alpha,3,r}(x)=O_{\alpha,r}(x^{1-\alpha/2})$ holds, 
then the growth rate $O_{\alpha,r}(x^{1-\alpha/2})$ is best in the following meaning.

\begin{proposition}\label{prop:k3}
	For all $\alpha\in(1,2)$ and $r\ge\mathbb{N}$, 
	\[
	\limsup_{x\to\infty} \frac{L_{\alpha,3,r}(x)}{x^{1-\alpha/2}}
	\ge \frac{\sqrt{2}-1}{\sqrt{\alpha(\alpha-1)r}}.
	\]
\end{proposition}

To prove Proposition~\ref{prop:k3}, 
we need to choose infinitely many $N\in\mathbb{N}$ with certain properties instead of \eqref{bestL}.
For this purpose, let us show the following lemmas.

\begin{lemma}\label{lem:lb1}
	Let $\alpha\in(1,2)$ and $r\in\mathbb{N}$.
	Then there exist infinitely many $N\in\mathbb{N}$ such that 
	$0\le\{f'(N)\}-1/2r<f''(N-1)$.
\end{lemma}
\begin{proof}
	Take an arbitrary $N\in\mathbb{N}$ such that $f''(N)<1/2r$ and $Nf''(2N)>1$.
	Since the inequality $f'(2N)-f'(N)>Nf''(2N)>1$ holds, 
	some $m\in\mathbb{Z}$ satisfies $f''(N) < 1/2r+m < f''(2N)$.
	Also, the sequence $(f'(N+n))_{n=0}^N$ increases and 
	the difference $f'(N+n+1)-f'(N+n)$ is bounded above by $f''(N)<1/2r$.
	Thus, we can take the minimum $n\in[1,N]_{\mathbb{Z}}$ such that $f'(N+n-1)<1/2r+m\le f'(N+n)<1+m$.
	Then it follows that 
	\[
	0 \le \{f'(N+n)\}-1/2r < f'(N+n)-f'(N+n-1) < f''(N+n-1).
	\]
	The arbitrariness of $N$ implies Lemma~\ref{lem:lb1}.
\end{proof}

\begin{lemma}\label{lem:lb2}
	Let $\alpha\in(1,2)$ and $r\in\mathbb{N}$.
	For all $c_0>2r^{1/2}$ and $c_1>r^{-1/2}$, 
	there exist infinitely many $N\in\mathbb{N}$ such that 
	$\{f(N)\}<c_1f''(N)^{1/2}$ and $0\le\{f'(N)\}-1/2r<c_0f''(N)^{1/2}$.
\end{lemma}
\begin{proof}
	Let $c_0>2r^{1/2}$ and $c_1>r^{-1/2}$.
	Take a sufficiently large $N\in\mathbb{N}$ such that 
	$0\le\{f'(N)\}-1/2r<f''(N-1)$ (see Lemma~\ref{lem:lb1}).
	Also, take $s\in[1,2r]_{\mathbb{Z}}$ such that $-1/2r<\{f(N)\}-s/2r\le0$.
	Defining $n_m=2rm-s$ and $x_m=f(N+n_m)-m-n_m\lfloor{f'(N)}\rfloor$ for $m\in[1,M+1]_{\mathbb{Z}}$, 
	we verify the following facts.
	\begin{enumerate}
		\item
		$0 < x_{m+1}-x_m < 2rf''(N-1) + 4r^2(M+1)f''(N)$ for all $m\in[1,M]_{\mathbb{Z}}$.
		\item
		$x_{M+1}-x_1 > 2r^2M^2f''(N+2r(M+1)) - 2r^2f''(N)$.
		\item
		$-1/2r < x_1 - \lfloor{f(N)}\rfloor < 2rf''(N-1) + 2r^2f''(N)$.
	\end{enumerate}
	Fact~(1): 
	\begin{align*}
		x_{m+1}-x_m
		&> 2rf'(N) - 1 - 2r\lfloor{f'(N)}\rfloor
		= 2r\{f'(N)\} - 1 \ge 0,\\
		x_{m+1} - x_m
		&< 2rf'(N+n_{m+1}) - 1 - 2r\lfloor{f'(N)}\rfloor\\
		&< 2r(f'(N) + n_{m+1}f''(N)) - 1 - 2r\lfloor{f'(N)}\rfloor\\
		&< 2r\{f'(N)\} - 1 + 2rn_{m+1}f''(N)\\
		&< 2rf''(N-1) + 4r^2(M+1)f''(N).
	\end{align*}
	Fact~(2): 
	\begin{align*}
		x_{M+1}-x_1
		&> f(N+n_{M+1}) - f(N+n_1) - M - 2rM\lfloor{f'(N)}\rfloor\\
		&= f(N+n_{M+1}) - f(N+n_1) - 2rMf'(N) + M(2r\{f'(N)\}-1)\\
		&\ge f(N+n_{M+1}) - f(N+n_1) - 2rMf'(N)\\
		&> \Bigl( f(N) + n_{M+1}f'(N) + \frac{n_{M+1}^2}{2}f''(N+n_{M+1}) \Bigr)\\
		&\quad- \Bigl( f(N) + n_1f'(N) + \frac{n_1^2}{2}f''(N) \Bigr) - 2rMf'(N)\\
		&= \frac{n_{M+1}^2}{2}f''(N+n_{M+1}) - \frac{n_1^2}{2}f''(N)\\
		&> 2r^2M^2f''(N+2r(M+1)) - 2r^2f''(N).
	\end{align*}
	Fact~(3): 
	\begin{align*}
		x_1 - \lfloor{f(N)}\rfloor
		&= f(N+n_1) - 1 - n_1\lfloor{f'(N)}\rfloor - \lfloor{f(N)}\rfloor\\
		&> f(N) + n_1f'(N) - 1 - n_1\lfloor{f'(N)}\rfloor - \lfloor{f(N)}\rfloor\\
		&= \{f(N)\} + n_1\{f'(N)\} - 1\\
		&\ge \{f(N)\} + n_1/2r - 1
		= \{f(N)\} - s/2r > -1/2r
	\end{align*}
	and 
	\begin{align*}
		x_1 - \lfloor{f(N)}\rfloor
		&= f(N+n_1) - 1 - n_1\lfloor{f'(N)}\rfloor - \lfloor{f(N)}\rfloor\\
		&< f(N) + n_1f'(N) + \frac{n_1^2}{2}f''(N) - 1 - n_1\lfloor{f'(N)}\rfloor - \lfloor{f(N)}\rfloor\\
		&= \{f(N)\} + n_1\{f'(N)\} + \frac{n_1^2}{2}f''(N) - 1\\
		&< \{f(N)\} + n_1(1/2r + f''(N-1)) + \frac{n_1^2}{2}f''(N) - 1\\
		&< \{f(N)\} - s/2r + n_1f''(N-1) + \frac{n_1^2}{2}f''(N)\\
		&\le n_1f''(N-1) + \frac{n_1^2}{2}f''(N)
		< 2rf''(N-1) + 2r^2f''(N).
	\end{align*}
	\par
	Now, we have the following two cases: 
	\begin{enumerate}
		\item
		$x_1 - \lfloor{f(N)}\rfloor\ge0$,
		\item
		$x_1 - \lfloor{f(N)}\rfloor<0$.
	\end{enumerate}
	Case~(1). The sufficiently large $N$ satisfies that 
	\begin{align*}
		\{f(N+n_1)\} &= \{x_1\} \overset{\text{Fact~(3)}}{<} 2rf''(N-1) + 2r^2f''(N) < c_0f''(N+n_1)^{1/2},\\
		\{f'(N+n_1)\} &< \{f'(N)\} + n_1f''(N)
		< 1/2r + f''(N-1) + 2rf''(N)\\
		&< 1/2r + c_1f''(N+n_1)^{1/2},\\
		\{f'(N+n_1)\} &> \{f'(N)\} \ge 1/2r.
	\end{align*}
	\par\noindent
	Case~(2). Take $1<\beta<\beta'=\min\{c_0/2r^{1/2}, c_1/r^{-1/2}\}$ and 
	put $M=\lceil{\beta f''(N)^{-1/2}/2r^{3/2}}\rceil=O(N^{1-\alpha/2})$.
	Since the sufficiently large $N$ satisfies 
	\begin{align*}
		x_{M+1}-x_1
		&> \frac{f''(N+2r(M+1))}{2rf''(N)} - 2r^2f''(N)\\
		&= \frac{\beta}{2r}\Bigl( \frac{N}{N+2r(M+1)} \Bigr)^{2-\alpha} - 2r^2f''(N)
		> \frac{1}{2r},
	\end{align*}
	we can take the minimum $m\in[1,M]_{\mathbb{Z}}$ such that $x_{m+1}-\lfloor{f(N)}\rfloor\ge0$.
	Then the sufficiently large $N$ satisfies that 
	\begin{align*}
		\{f(N+n_{m+1})\} &= \{x_{m+1}\} < x_{m+1}-x_m
		\overset{\text{Fact~(2)}}{<} 2rf''(N-1) + 4r^2(M+1)f''(N)\\
		&< 2r^{1/2}\beta'f''(N+n_{m+1})^{1/2}
		\le c_0f''(N+n_{m+1})^{1/2}
	\end{align*}
	and 
	\begin{align*}
		\{f'(N+n_{m+1})\} &< \{f'(N)\} + n_{m+1}f''(N)
		< 1/2r + f''(N-1) + 2r(M+1)f''(N)\\
		&< 1/2r + r^{-1/2}\beta'f''(N+n_{m+1})^{1/2}
		\le 1/2r + c_1f''(N+n_{m+1})^{1/2},\\
		\{f'(N+n_{m+1})\} &> \{f'(N)\} \ge 1/2r.
	\end{align*}
	Therefore, Lemma~\ref{lem:lb2} holds.
\end{proof}

\begin{proof}[Proof of Proposition~\upshape{\ref{prop:k3}}]
	Let $c_0>2r^{1/2}$, $c_1>r^{-1/2}$ and $0<c_2<\sqrt{c_1^2+1/r}-c_1$.
	Thanks to Lemma~\ref{lem:lb2}, 
	we can take a sufficiently large $N\in\mathbb{N}$ such that 
	\begin{enumerate}
		\item
		$\{f(N)\}<c_0f''(N)^{1/2}$,
		\item
		$0\le\{f'(N)\}-1/2r<c_1f''(N)^{1/2}$,
		\item
		$rc_1f''(N)^{1/2}<1/2$.
	\end{enumerate}
	Moreover, the inequality 
	\begin{enumerate}
		\addtocounter{enumi}{3}
		\item
		$0\le\{rf'(N)\}-1/2<rc_1f''(N)^{1/2}$
	\end{enumerate}
	follows from (2) and (3).
	Set $\epsilon=\epsilon(N)=2r^2f''(N)\in(0,1)$, 
	which is just \eqref{eqapp2} with $k=3$.
	We show that $L_{\alpha,3,r}(N) > c_2f''(N)^{-1/2}$ by contradiction.
	Suppose that $L_{\alpha,3,r}(N) \le c_2f''(N)^{-1/2}$.
	Take $m\in[0,L_{\alpha,3,r}(N)]_{\mathbb{Z}}$ such that 
	$(\lfloor{f(N+m+rj)}\rfloor)_{j=0}^2$ is an AP.
	Since the set $\cC_{k,2}^{+}(\epsilon)$ defined by \eqref{eq20} 
	satisfies the inclusion relation \eqref{eq12}, 
	the vector $(\{f(N+m)\}, \{rf'(N+m)\}+s_1)$ lies in $\cC_{k,2}^{+}(\epsilon)$ for some $s_1\in\mathbb{Z}$.
	The integer $s_1$ is equal to $0$ or $-1$ (see the end of the proof of Proposition~\ref{prop:k4}).
	\par
	If $s_1=0$, then the inequalities $-\epsilon\le \{f(N+m)\}+2\{rf'(N+m)\}<1$, 
	$m \le L_{\alpha,3,r}(N) \le c_2f''(N)^{-1/2}$ and (4) yield that 
	\[
	\{rf'(N+m)\} < 1/2 \le \{rf'(N)\} < 1/2 + rc_1f''(N)^{1/2}
	\]
	and thus 
	\[
	1/2 - rc_1f''(N)^{1/2} < rf'(N+m) - rf'(N)
	\le rmf''(N) \le rc_2f''(N)^{1/2},
	\]
	which is a contradiction because $N$ is sufficiently large.
	\par
	Next, consider the case $s_1=-1$.
	Then the inequalities $-\epsilon\le \{f(N+m)\}+2(\{rf'(N+m)\}-1)<1$, 
	$m \le L_{\alpha,3,r}(N) \le c_2f''(N)^{-1/2}$ and (4) yield that 
	\begin{align*}
		1-\frac{\{f(N+m)\}+\epsilon}{2} &\le \{rf'(N+m)\}
		\le \{rf'(N)\} + rmf''(N)\\
		&< 1/2 + rc_1f''(N)^{1/2} + rc_2f''(N)^{1/2},
	\end{align*}
	whence 
	\begin{equation}
		\{f(N+m)\} > 1-\epsilon-2r(c_1+c_2)f''(N)^{1/2}.
		\label{eq21}
	\end{equation}
	Since Taylor's theorem implies that 
	\[
	f(N+m) = f(N) + mf'(N) + \frac{m^2}{2}f''(N+\theta)
	\]
	for some $\theta\in[0,m]$, 
	the inequalities (1) and $m \le L_{\alpha,3,r}(N) \le c_2f''(N)^{-1/2}$ yield that 
	\begin{equation}
	\begin{split}
		\{f(N+m)\} &\le \{f(N)\} + \{mf'(N)\} + \frac{m^2}{2}f''(N+\theta)\\
		&< c_0f''(N)^{1/2} + \{mf'(N)\} + c_2^2/2.
	\end{split}\label{eq22}
	\end{equation}
	Also, the inequalities (2) and $m \le L_{\alpha,3,r}(N) \le c_2f''(N)^{-1/2}$ yield that 
	\[
	0 \le m\{f'(N)\}-m/2r < c_1mf''(N)^{1/2} \le c_1c_2,
	\]
	whence 
	\begin{equation}
		\{mf'(N)\} \le \{m/2r\} + c_1c_2 \le 1-1/2r+c_1c_2.
		\label{eq23}
	\end{equation}
	Recall the definition of $\epsilon=\epsilon(N)$.
	Using \eqref{eq21}--\eqref{eq23}, we have 
	\begin{align*}
		&\quad 1-2r^2f''(N)-2r(c_1+c_2)f''(N)^{1/2}
		< \{f(N+m)\}\\
		&< c_0f''(N)^{1/2} + (1-1/2r+c_1c_2) + c_2^2/2,
	\end{align*}
	whence 
	\begin{equation}
		1/2r-c_1c_2-c_2^2/2 < 2r^2f''(N) + (c_0+2r(c_1+c_2))f''(N)^{1/2}.
		\label{eqapp4}
	\end{equation}
	Since the assumption $0<c_2<\sqrt{c_1^2+1/r}-c_1$ implies $1/2r-c_1c_2-c_2^2/2>0$, 
	the inequality \eqref{eqapp4} is a contradiction because $N$ is sufficiently large.
	Therefore, 
	\[
	\limsup_{x\to\infty} \frac{L_{\alpha,3,r}(x)}{x^{1-\alpha/2}}
	\ge \frac{c_2}{\sqrt{\alpha(\alpha-1)}}.
	\]
	Finally, letting $c_2\to\sqrt{c_1^2+1/r}-c_1$ and $c_1\to r^{-1/2}$, 
	we obtain Proposition~\ref{prop:k3}.
\end{proof}

Finally, let us show the following proposition that supports $L_{\alpha,3,r}(x)=O_{\alpha,r}(x^{1-\alpha/2})$.

\begin{proposition}\label{prop:k3dense}
	Let $\alpha\in(1,2)$ and $r\in\mathbb{N}$, 
	and let $w(x)$ be an arbitrary positive-valued function such that 
	$x^{\alpha/2-1}w(x)\to0$ and $w(x)\to\infty$ as $x\to\infty$.
	Then 
	\[
	\lim_{M\to\infty} \frac{\#\{ N\in[1,M]_{\mathbb{Z}} : L_{\alpha,3,r}(N) \leq N^{1-\alpha/2}w(N) \}}{M}
	= 1.
	\] 
\end{proposition}
\begin{proof}
	For $N,L\in\mathbb{N}$, 
	define $D(N,L)$ as the discrepancy of the sequence $(f(n))_{n=N}^{N+L-1}$.
	Let $L=L(N)=\lceil{N^{1-\alpha/2}w(N)}\rceil$ and $H=H(N)=\lceil{N^{(2-\alpha)/3}}\rceil$.
	Lemma~\ref{ETK} with $d=1$ and Lemma~\ref{2ndderiv} imply that for every $N\in\mathbb{N}$, 
	\begin{align*}
		D(N,L) &\ll \frac{1}{H} + \frac{1}{L}\sum_{h=1}^H \frac{1}{h}\abs{\sum_{n=N}^{N+L-1} e(hf(n))}\\
		&\ll_{\alpha,w(\cdot)} 1/H + H^{1/2}N^{\alpha/2-1} + N^{1-\alpha/2}/L
		\ll N^{(\alpha-2)/3} + 1/w(N).
	\end{align*}
	Thus, there exists $C>0$ such that for every $N\in\mathbb{N}$, 
	\[
	D(N,L) \le C(N^{(\alpha-2)/3} + 1/w(N)).
	\]
	Now, let $\epsilon\in(0,1/6)$ be arbitrary.
	Define the sets $\cC_{3,2}^{-}(\epsilon)$, $\cV_0$, $\cV_1$ and $\cV$ as 
	\begin{align*}
		\cC_{3,2}^{-}(\epsilon) &= \{ (y_0,y_1)\in\mathbb{R}^2 : 0\le y_0<1,\ 0\le y_0+2y_1<1-\epsilon \},\\
		\cV_0 &= \bigl\{ N\in\mathbb{N} : \{rf'(N)\} < 1/2-3\epsilon \bigr\},\\
		\cV_1 &= \bigl\{ N\in\mathbb{N} : 1/2+\epsilon < \{rf'(N)\} < 1-\epsilon \bigr\},\\
		\cV &= \{ N\in\mathbb{N} : L_{\alpha,3,r}(N) \leq N^{1-\alpha/2}w(N) \}.
	\end{align*}
	Due to the assumptions $x^{\alpha/2-1}w(x)\to0$ and $w(x)\to\infty$, 
	we can taking a positive number $x_0$ such that 
	\begin{enumerate}
		\item
		$C(x^{(\alpha-2)/3} + 1/w(x)) < 2\epsilon$ for all $x\ge x_0$,
		\item
		$r\alpha(\alpha-1)x^{\alpha/2-1}w(x) < \epsilon$ for all $x\ge x_0$,
		\item
		$2r^2f''(x)\le\epsilon$ for all $x\ge x_0$.
	\end{enumerate}
	Let us show the inclusion relation $(\cV_0\cup\cV_1)\cap[x_0,\infty)\subset\cV$ below.
	\par
	First, assume $N\in\cV_0\cap[x_0,\infty)$.
	Then the set $\cW_0 \coloneqq \{ n\in[0,L)_{\mathbb{Z}} : \epsilon<\{f(N+n)\}<3\epsilon \}$ satisfies 
	\[
	\#\cW_0/L \ge 2\epsilon - D(N,L)
	= 2\epsilon - C(N^{(\alpha-2)/3} + 1/w(N)) \overset{(1)}{>} 0.
	\]
	Take an element $m\in\cW_0\neq\emptyset$.
	Then the assumption $N\in\cV_0$ implies that  
	\begin{align*}
		\{rf'(N+m)\} &\leq \{rf'(N)\} + rmf''(N)
		< 1/2-3\epsilon + r\alpha(\alpha-1)(L-1)N^{\alpha-2}\\
		&< 1/2-3\epsilon + r\alpha(\alpha-1)N^{\alpha/2-1}w(N)
		\overset{(2)}{<} 1/2-2\epsilon.
	\end{align*}
	Thus, 
	\[
	0 \le \{f(N+m)\}+2\{rf'(N+m)\} < 3\epsilon + 2(1/2-2\epsilon) = 1-\epsilon,
	\]
	whence $(\{f(N+m)\}, \{rf'(N+m)\})\in\cC_{3,2}^{-}(\epsilon)$.
	Therefore, $(\lfloor{f(N+m+rj)}\rfloor)_{j=0}^2$ is an AP (see the proof of Proposition~\ref{main0'}).
	Since the inequality $L_{\alpha,3,r}(N) \le m < L$ holds, 
	it turns out that $N$ lies in $\cV$.
	\par
	Next, assume $N\in\cV_1\cap[x_0,\infty)$.
	The set $\cW_1 \coloneqq \{ n\in[0,L)_{\mathbb{Z}} : 1-2\epsilon<\{f(N+n)\}<1-\epsilon \}$ is also not empty 
	in the same way as $\cW_0\neq\emptyset$.
	Take an element $m\in\cW_1$.
	Since the difference $rf'(N+m) - rf'(N)$ is bounded above by 
	\[
	rmf''(N) \le r\alpha(\alpha-1)(L-1)N^{\alpha-2}
	< r\alpha(\alpha-1)N^{\alpha/2-1}w(N) \overset{(2)}{<} \epsilon,
	\]
	the assumption $N\in\cV_1$ implies $\{rf'(N+m)\} \ge \{rf'(N)\} > 1/2+\epsilon$.
	This and $1-2\epsilon<\{f(N+m)\}<1-\epsilon$ yield that 
	\begin{gather*}
		\{f(N+m)\} + 2(\{rf'(N+m)\}-1) < 1-\epsilon,\\
		\{f(N+m)\} + 2(\{rf'(N+m)\}-1)
		> 1-2\epsilon + 2(1/2+\epsilon-1) = 0,
	\end{gather*}
	whence $(\{f(N+m)\}, \{rf'(N+m)\}-1)\in\cC_{3,2}^{-}(\epsilon)$.
	Therefore, $(\lfloor{f(N+m+rj)}\rfloor)_{j=0}^2 $ is an AP (see the proof of Proposition~\ref{main0'}).
	Since the inequality $L_{\alpha,3,r}(N) \le m < L$ holds, 
	it turns out that $N$ lies in $\cV$.
	\par
	The inclusion relation $(\cV_0\cup\cV_1)\cap[x_0,\infty)\subset\cV$ has been proved above.
	Since the sequence $(rf'(N))_{N=1}^\infty$ is uniformly distributed modulo $1$ 
	and the sets $\cV_0$ and $\cV_1$ are disjoint, 
	it follows that 
	\begin{align*}
		\liminf_{M\to\infty} \frac{\#(\cV\cap[1,M])}{M}
		&\geq \liminf_{M\to\infty} \frac{\#(\cV_0\cap[1,M])}{M}
		+ \liminf_{M\to\infty} \frac{\#(\cV_1\cap[1,M])}{M}\\
		&\geq (1/2-3\epsilon) + (1/2-2\epsilon)
		= 1-5\epsilon.
	\end{align*}
	Letting $\epsilon\to+0$, 
	we obtain Proposition~\ref{prop:k3dense}.
\end{proof}

\end{document}